\documentclass[11pt]{article}
\usepackage[oldstylenums]{kpfonts}
\usepackage{lettrine}
\usepackage{geometry}
\usepackage{tikz}
\usetikzlibrary{arrows}
\usetikzlibrary{patterns}
\usetikzlibrary{hobby}
\usetikzlibrary{decorations.pathreplacing}
\usetikzlibrary{decorations.markings}
\usepackage{array}
\usepackage{graphicx}
\usepackage{amsmath,amsfonts,amssymb,amsthm,mathrsfs}
\usepackage{hyperref}
\usepackage[active]{srcltx}
\usepackage[T1]{fontenc}
\usepackage{color}
\usepackage[makeroom]{cancel}
\usepackage{enumitem}

\numberwithin{equation}{section}
\newtheorem{theo}{Theorem}[section]
\newtheorem{hyp}[theo]{Assumption}

\newtheorem{lem}[theo]{Lemma}

\newtheorem{cor}[theo]{Corollary}
\newtheorem{prop}[theo]{Proposition}
\newtheorem{rmk}[theo]{Remark}
\newtheorem{setting}[theo]{Setting}

\allowdisplaybreaks

\newcommand{\fr}{\penalty-20\null\hfill$\blacksquare$}         

\newcommand{\eps}{\varepsilon}
\newcommand{\qtext}[1]{\quad\mbox{#1}\quad} 
\newcommand{\qqtext}[1]{\qquad\mbox{#1}\qquad} 
\newcommand{\R}{\mathbb{R}}

\newcommand{\g}{\gamma}

\newcommand{\supp}{\operatorname{supp}}

\renewcommand{\P}{\mathcal{P}}

\renewcommand{\o}{\omega}

\renewcommand{\S}{\mathsf{S}} 
\renewcommand{\d}{{\mathrm d}}			
\newcommand{\dt}{{\d t}}
\newcommand{\ddt}{{\frac \d\dt}}

\newcommand{\sfd}{{\sf d}}				
\newcommand{\X}{{\rm X}}				
\newcommand{\V}{{\sf E}}
\newcommand{\EVI}{{\rm EVI}}

\newcommand{\RCD}{{\rm RCD}}
\newcommand{\CAT}{{\rm CAT}}
\newcommand{\m}{\mathfrak{m}}			
\newcommand{\cost}{\mathscr{C}}		

\begin{document}

\title{The dynamical Schr\"odinger problem in abstract metric spaces}
\date{}
\author{L\'eonard Monsaingeon
    \thanks{GFM Universidade de Lisboa, Campo Grande, Edif\'icio C6, 1749-016 Lisboa, Portugal
    and IECL Universit\'e de Lorraine, F-54506 Vandoeuvre-l\`es-Nancy Cedex, France.
    email: leonard.monsaingeon@univ-lorraine.fr
}
\and Luca Tamanini
    \thanks{CEREMADE (UMR CNRS 7534), Universit\'e Paris Dauphine PSL, Place du Mar\'echal de Lattre de Tassigny, 75775 Paris Cedex 16, France
    and INRIA-Paris, MOKAPLAN, 2 Rue Simone Iff, 75012, Paris, France.
    email: tamanini@ceremade.dauphine.fr}
\and Dmitry Vorotnikov
    \thanks{University of Coimbra, CMUC, Department of Mathematics, 3001-501 Coimbra, Portugal.
    email: mitvorot@mat.uc.pt
    }
    }

\maketitle

\begin{abstract}
In this paper we introduce the dynamical Schr\"odinger problem, defined for a wide class of entropy and Fisher information functionals, as a geometric problem on abstract metric spaces.
Under very mild assumptions we prove a generic $\Gamma$-convergence result towards the geodesic problem as the noise parameter $\eps\downarrow 0$.
We also study the dependence of the entropic cost on the parameter $\eps$.
Some examples and applications are discussed.
\end{abstract}

\noindent
MSC [2020] 49J45, 49Q20, 58B20.\\
Keywords: gradient flows, metric geometry, optimal transport, 
Fisher information, Gamma-convergence

\tableofcontents


\section{Introduction}

Gaspard Monge and Erwin Schr\"odinger came up with two a priori unrelated problems that are concerned with finding a preeminent way of deforming a prescribed probability distribution into another one. While Monge was interested in optimizing the cost of transportation of goods \cite{Villani03,Villani09,S15}, Schr\"odinger's original thought experiment \cite{Schrodinger31, Schrodinger32} aimed for finding the most likely evolution between two subsequent observations of a cloud of independent particles. 
So, even if in both cases we are facing an interpolation and optimization problem, the former is deterministic in nature whereas the latter is strongly related to large deviations theory, and is, at the first glance, purely stochastic. We refer to a recent survey \cite{CGP16} for various formulations and aspects of the Schr\"odinger problem, and to \cite{Zambrini86,Z15} for a discussion of its role in Euclidean Quantum Mechanics. 

Anyway, several analogies and connections exist between the two problems.
They can be appreciated by looking carefully at the interpolation aspects of both problems, namely at their dynamical formulations and the underlying equations governing the respective evolutions. 
In the case of a quadratic transportation cost over a Riemannian manifold $M$, the Monge-Kantorovich optimal transport problem is solved (at least in a weak sense) by interpolating between the source and the target distributions with a constant-speed, length-minimizing geodesic in the Otto-Wasserstein space of probability measures $\mathcal P_2(M)$.
This gives a curve $(\mu_t)_{t \in [0,1]} \subset \mathcal P_2(M)$ which formally satisfies (in the sense of the celebrated Otto calculus \cite{Otto01,Villani03,Villani09})
\begin{equation}\label{eq:otto-geod}
\nabla_{\dot \mu_t}\dot \mu_t = 0,
\end{equation}
where $\nabla_{\dot\mu_t}$ is the covariant derivative along the curve $t \mapsto \mu_t$. 
The Schr\"odinger problem with parameter $\eps$ (from a physical viewpoint, $\eps$ can be seen as a temperature or level of noise) can also be translated into such a geometric language.
By analogy, when looking at the covariant derivative along the optimal evolution $(\mu_t^\eps)_{t \in [0,1]}$, usually called Schr\"odinger bridge or entropic interpolation, the resulting equation is surprising and can be viewed \cite{Conforti17} as Newton's second law
\begin{equation}\label{eq:otto-newton}
\nabla_{\dot \mu_t^\eps}\dot \mu_t^\eps = \frac {\eps^2} 8 \nabla I(\mu_t^\eps),
\end{equation}
where in the right-hand side $\nabla$ denotes the gradient in the Otto-Wasserstein pseudo-Riemannian sense and $I$ is the Fisher information
$$
I(\mu)=4 \int_M |\nabla \sqrt \rho|^2\,\d\mathrm{vol}=\int_M |\nabla\log\rho|^2\rho\,\d\mathrm{vol}
$$
provided $\mu = \rho\cdot\mathrm{vol}$.
A related observation is that the (scaled) heat flow, coinciding with the (scaled) gradient flow of the Boltzmann-Shannon entropy \cite{JKO98, Villani09}
$$
H(\mu)=\int_M \rho \log \rho\,\d\mathrm{vol}
$$
for $\mu = \rho\cdot \mathrm{vol}$, is also a solution to \eqref{eq:otto-newton}: a simple differentiation in time of $\dot\mu_t=-\frac{\eps}{2}\nabla H(\mu_t)$ and the fact that $I = |\nabla H|^2$ in the Otto-Wasserstein sense automatically yield 
\[
\nabla_{\dot\mu_t}\dot\mu_t = \frac{\eps}{2} \nabla^2 H(\mu_t)\cdot\frac{\eps}{2}\nabla H(\mu_t) = \frac{\eps^2}{8}\nabla |\nabla H(\mu_t)|^2 = \frac{\eps^2}{8}\nabla I(\mu_t).
\]
This shows that the Schr\"odinger problem lies between optimal transport and diffusion and is naturally intertwined with both deterministic behaviour and Brownian motion. It shares the same Newton's law as the gradient flow of the entropy, but unlike the heat flow it has a prescribed final configuration to match:
it is up to the parameter $\eps$ to tip the balance in favour of deterministic transport or diffusion.
With this heuristics in mind, we see that as $\eps \to 0$ the applied force $\eps^2 \nabla I(\mu_t^\eps)$ in \eqref{eq:otto-newton} vanishes, so that the Schr\"odinger problem may be interpreted as a noisy (entropic) counterpart of the Monge-Kantorovich optimal transport, corresponding to the unforced geodesic evolution \eqref{eq:otto-geod} discussed above.
This informal relationship has a rigorous counterpart, which dates back to the pioneering works on the asymptotic behavior of the Schr\"odinger problem as $\eps\to 0$ of T.\ Mikami, M.\ Thieullen \cite{Mikami04,MikamiThieullen08}, and C.\ L\'eonard \cite{Leonard12,Leonard14}. 
This was subsequently developed in \cite{CDPS17,BaradatMonsaingeon18,GigTam18}.
Very recently \cite{LoM20,MG20}, similar small-noise results were obtained for static Monge-Kantorovich problems regularized with more general entropies. 

This first connection can be investigated further and by doing so one can remark that \eqref{eq:otto-newton} is exactly the Euler-Lagrange optimality equation for the dynamical Benamou-Brenier formulation of the Schr\"odinger problem \cite{C14,CGP16,Leonard14,GigTam20}, which consists in minimizing the Lagrangian kinetic action perturbed by the Fisher information:
In more precise terms,
\begin{equation}\label{eq:bbs}
\inf\bigg\{\frac{1}{2}\int_0^1\int_M |v_t|^2\,\d\mu_t\d t + \frac{\eps^2}{8}\int_0^1 I(\mu_t)\,\d t\bigg\},
\end{equation}
where the infimum runs over all solutions of the continuity equation
$$
\partial\mu_t + \mathrm{div}(v_t\mu_t)=0
$$
with prescribed initial and final densities.
Also from this variational standpoint the reader can see that as $\eps \to 0$ the Schr\"odinger problem formally reduces to
\begin{equation} \label{BB}
\inf\frac{1}{2}\int_0^1\int_M |v_t|^2\,\d\mu_t\d t,
\end{equation}
namely to the dynamical Benamou-Brenier formulation of the (quadratic) optimal transport problem \cite{BenamouBrenier00}. This variational representation depicts in a way clearer than \eqref{eq:otto-newton} the double nature of the Schr\"odinger problem, the competition between the determinism encoded in the kinetic energy and the unpredictability coming from the Fisher information, and the role played by $\eps$ in balancing these two opposite behaviours.
\bigskip

The double bond of the Schr\"odinger problem with optimal transport on the one hand and heat flow on the other hand results in fruitful and wide-ranging applications of both theoretical and applied interest.
Indeed, from the connection with the heat flow the solutions to the Schr\"odinger problem gain regularity properties which are not available in optimal transport, and thanks to the asymptotic behaviour of the Schr\"odinger problem as $\eps \to 0$ entropic interpolations represent an efficient way to approximate Wasserstein geodesics with second-order accuracy \cite{GigTam18,ConTam19}.
This approach has already turned out to be successful in conjunction with functional inequalities \cite{ConRip18,GLRT19} and differential calculus along Wasserstein geodesics \cite{GigTam18}.
But the nice behaviour of Schr\"odinger bridges is important also for computational purposes.
The impact of Schr\"odinger problem and Sinkhorn algorithms (deeply related to the static formulation of the former) on the numerical methods used in optimal transportation theory has been impressive, as witnessed by several recent works (see \cite{peyre2019computational} and references therein as well as \cite{Cuturi13, BCCNP15, BCN16, BCN17, BCDMN18,MG19}).
\bigskip

As a matter of fact neither the particular structure of the Wasserstein space nor the specific choice of the Boltzmann-Shannon functional are required to define the two problems in question (cf.\ a related discussion in the heuristic paper \cite{leger2019geometric}): one can of course define length-minimizing geodesics in any metric space $(\X,\sfd)$, and the Schr\"odinger problem (or at least its Benamou-Brenier formulation described above) merely involves an entropy functional and a corresponding Fisher information.
Given such a reasonable entropy functional $\V$ on $\X$ that generates a gradient flow in a suitable sense, the corresponding Fisher information is expected to be nothing but the dissipation rate of $\V$ (along solutions of its own gradient flow), just as $I$ coincides with the rate of dissipation of the entropy $H$ along the heat flow.
This observation is the starting point of the present paper, where we intend to study the abstract Schr\"odinger bridge problem or, in other words, the entropic approximation of geodesics in metric spaces. 

The first main result that we achieve is the \textbf{$\Gamma$-convergence (Theorems \ref{t:gconv1} and \ref{t:gconv2})}. Under very mild assumptions on $\X$ and $\V$, we will prove the solvability of the abstract $\eps$-Schr\"odinger problem and the $\Gamma$-convergence to the corresponding geodesic problem as $\eps\to 0$.
We will also rigorously justify, in the metric setting, that any trajectory of a gradient flow solves an associated Schr\"odinger problem (Proposition \ref{p:gfs}).
Leveraging a quantitative $AC^2$ estimate based on a straightforward chain-rule in the smooth Riemannian setting, the cornerstone of our analysis will be the systematic construction of an $\eps$-regularized entropic copy $(\gamma^\eps_t)_{t\in[0,1]}$ of any arbitrary curve $(\gamma_t)_{t\in[0,1]}$.
These perturbed curves will provide recovery sequences for the $\Gamma$-convergence.
Our construction is completely Eulerian and essentially consists in running the $\V$-gradient flow for a short time $h_\eps(t)$ starting at $\gamma_t$ for all $t$, for well-chosen functions $h_\eps\geq 0$.
The challenge here will be to reproduce the (formal, differential) Riemannian chain-rule in metric spaces.
This idea of perturbing curves ``in the direction of the gradient flow'' appeared first in \cite{erbar2015large} in a slightly different context.
Notably and more recently, this approach has also been used independently by A. Baradat and some of the authors \cite{BaradatMonsaingeon18,MV20} in order to prove the $\Gamma$-convergence for the classical dynamical Schr\"odinger problem on the Otto-Wasserstein space and for its counterpart on the non-commutative Fisher-Rao space, respectively.
However, in those papers the computations were {\it ad hoc} and heavily exploited the underlying structures of the particular spaces as well as the properties of the particular gradient flows (namely, of the classical heat flow and of its restriction to multivariate Gaussians), whereas here we derive everything from the existence of an abstract gradient flow on $\X$ driven by $\V$.

\begin{rmk} {\rm In the smooth Riemannian setting, and given $\lambda\in \R$, elementary calculus shows that the $\lambda$-convexity of $\V$ along geodesics is of course equivalent to a uniform lower bound $\mathrm{Hess}\,\V(x) \geq \lambda \operatorname{Id}$ as quadratic forms in the tangent space, but also more importantly to the $\lambda$-contractivity of the $\V$-gradient flow.
		In the metric setting no second order calculus is available in general, and the very notion of gradient flow as well as its connection with geodesic convexity and contractivity become much more subtle.
		The key notion of gradient flow that we shall use throughout is that of \emph{Evolution Variational Inequality}, or $\EVI_\lambda$ flow \cite{AmbrosioGigliSavare08}.
		Under reasonable assumptions it is well known that (a suitable variant of) convexity of $\V$ generally provides existence of an $\EVI_\lambda$-flow starting at any $x\in \X$, see \cite{AmbrosioGigliSavare08}.
		A natural question to ask is whether the converse also holds true, i.e.\ whether well-posedness of a reasonable gradient flow implies some convexity.
		This was proved in \cite{bolley2014nonlinear} for the specific case of the Euclidean Wasserstein space $\X=W_2(\Omega)$, $\Omega\subset\R^d$, and at least for the so-called internal energies, and it is shown therein that $0$-contractivity of the gradient flow (or equivalently, of the associated nonlinear diffusion equation) implies $0$-displacement convexity in the sense of McCann \cite{McC97}.
		In the same spirit, and building up on Otto and Westdickenberg \cite{otto2005eulerian}, Daneri and Savar\'e proved in a very general metric setting that the generation of an $\EVI_\lambda$-flow indeed implies $\lambda$-geodesic convexity \cite[Theorem 3.2]{DaneriSavare08}.
		A byproduct of our analysis for the $\Gamma$-convergence will give a new independent proof of this latter fact by a completely different approach, essentially by constructing an $\eps$-entropic regularization of geodesics and carefully examining the defect of optimality at order one in $\eps\to 0$.}
\end{rmk}

The second group of main results has to do with the \textbf{behaviour of the cost and its Taylor expansion (Proposition \ref{prop:derivative} and Theorem \ref{thm:taylor}).} As a main application of the $\Gamma$-convergence of the Schr\"odinger problem to the geodesic problem as $\eps \to 0$ (and more generally of the $\eps'$-Schr\"odinger problem to the $\eps$-one as $\eps' \to \eps$) we investigate the behaviour of the optimal value of the dynamical Schr\"odinger problem, henceforth called \emph{entropic cost}, as a function of the temperature parameter $\eps$, with particular emphasis on the regularity and the small-noise regime.
For the classical dynamical Schr\"odinger problem \eqref{eq:bbs}, it has recently been proved by the second author with G.\ Conforti \cite{ConTam19} that the entropic cost is of class $C^1((0,\infty)) \cap C([0,\infty))$ (actually $C^1([0,\infty))$ under suitable assumptions) and twice a.e.\ differentiable; once this regularity information is available, the formula for the first derivative is rather easy to guess, as by the envelope theorem it coincides with the partial derivative w.r.t.\ $\eps$ of the functional in \eqref{eq:bbs} evaluated at any critical point. Denoting by $\cost_\eps(\mu,\nu)$ the value in \eqref{eq:bbs} with marginal constraints $\mu$ and $\nu$ and by $(\mu_t^\eps)_{t \in [0,1]}$ the associated Schr\"odinger bridge, this statement reads as 
\[
\frac{\d}{\d\eps}\cost_\eps(\mu,\nu) = \frac{\eps}{4}\int_0^1 I(\mu_t^\eps)\,\d t, \qquad \forall \eps > 0
\]
and in \cite{ConTam19} this identity played an important role in the study of both the large- and small-noise behaviour of the Schr\"odinger problem, obtaining in particular a Taylor expansion around $\eps=0$ with $o(\eps^2)$-accuracy.
Since the central object in the present paper is an abstract and general formulation of \eqref{eq:bbs}, an analogous result is expected to hold. However, from a technical viewpoint the proof is much more subtle and challenging, because unlike \eqref{eq:bbs} our metric version of the dynamical Schr\"odinger problem may have multiple solutions. 
For this reason the discussion about the regularity of the entropic cost in this paper is less concise than in \cite{ConTam19}.
Nonetheless, we are still able to deduce the same kind of Taylor expansion with the same accuracy.
Given the previous interpretation of the Schr\"odinger problem as a noisy Monge-Kantorovich problem and the importance of quantitative estimates in approximating optimal transport by means of the Schr\"odinger problem, it is reasonable to expect that such a Taylor expansion (valid in a general framework for a wide choice of functionals $\V$) will fit to a countless variety of examples, some of which will be discussed here.

\paragraph{Structure of the paper.}
In Section~\ref{sec:heuristics} we give a short and formal proof of our fundamental $AC^2$ estimate in the smooth Riemannian setting, and show how it can be exploited to establish $\Gamma$-convergence and convexity.
Section~\ref{sec:metric spaces} fixes the metric framework in which we work for the rest of the paper, and extends the previous estimate to this metric setting.
In Section~\ref{sec:small_noise_convexity} we prove the $\Gamma$-convergence as $\eps\downarrow 0$, and offer a new proof of the geodesic convexity of the generators of $\EVI$-flows.
Section~\ref{sec:derivative_cost} studies the dependence of the optimal entropic cost on the temperature parameter $\eps>0$, and provides a second order expansion. 
Finally, we list in Section~\ref{s: exam} several examples and applications covered by our abstract results.


\section{Heuristics}\label{sec:heuristics}

Here we remain formal and the computations are carried in a Riemannian setting, where classical calculus and chain-rules are available.
(Significant work will be required later on to adapt the computations in metric spaces.)
All the objects and functions in this section are therefore considered to be smooth, and we deliberately ignore any regularity issue.

Let $M$ be a Riemannian manifold with scalar product $\langle.,.\rangle_q$ at a point $q\in M$ and induced Riemannian distance $d$, and let $V:M\to \R$ be a given potential.
For simplicity we assume here that $V$ is globally bounded from below on $M$, and up to replacing $V$ by $V-\min V$ we can assume that $V(q)\geq 0$.
(In section~\ref{sec:metric spaces} we will relax this assumption and allow $V$ to be only \emph{locally bounded from below}.)
Given a small temperature parameter $\eps>0$, and following \cite{leger2019geometric}, the (dynamic) \emph{geometric Schr\"odinger problem} consists in solving the optimization problem
\begin{multline}
 \label{eq:schrodinger_pb_heuristic}
 \frac{1}{2}\int_0^1\left|\frac{\d q_t}{\d t}\right|^2\,\dt + \frac{\eps^2}{2}\int_0^1 |\nabla V|^2(q_t)\,\dt
\qquad 
 \longrightarrow 
 \qquad 
 \min;
 \\
 \mbox{s.t. }q\in C([0,1],M)
 \mbox{ with endpoints }q_0,q_1.
\end{multline}
For $s\geq 0$ we denote by $\Phi(s,q_0)$ the semi-flow corresponding to the autonomous $V$-gradient flow started from $q_0\in M$,
\[
\left\{
\begin{array}{l}
 \displaystyle{\frac{\d }{\d s}\Phi(s,q_0)=-\nabla V(\Phi(s,q_0))},\\
 \Phi(0,q_0) = q_0.
\end{array}
\right.
\]
The goal of this section is to give a straightforward proof of the following two facts, assuming that the potential $V$ is well behaved:
\begin{enumerate}[label=({\roman*})]
 \item 
 the $\eps$-Schr\"odinger problem converges to the geodesic problem as $\eps\to 0$;
 \item
 $\lambda$-contractivity of the generated flow $\Phi$ can be turned into $\lambda$-convexity along geodesics.
\end{enumerate}
With this goal in mind, fix any two endpoints $q_0,q_1\in M$ and take an arbitrary curve joining them
$$
q\in C([0,1],M),
\hspace{1.5cm}
q|_{t=0}=q_0\qtext{and}q|_{t=1}=q_1.
$$
For any  function $h(t)\geq 0$ with $h(0)=h(1)=0$, we perturb $q$ by defining
$$
\tilde q_t:=\Phi(h(t),q_t),
\hspace{1.5cm}
t\in[0,1]
$$
i.e.\ $\tilde q_t$ is the solution of the $V$-gradient flow at time $s=h(t)\geq 0$ starting from $q_t$ at time $s=0$.
We shall refer to $t\in[0,1]$ as a ``horizontal time'' and to $s\in[0,h(t)]$ as a ``vertical time'', see Figure~\ref{fig:qtilde}.
Later on we will think of the curve $\tilde q$ as a ``regularized'' version of $q$.

\begin{figure}[h!]
  \begin{center}
  \begin{tikzpicture}
  [
nodewitharrow/.style 2 args={                
            decoration={             
                        markings,   
                        mark=at position {#1} with { 
                                    \arrow{stealth},
                                    \node[transform shape,above] {#2};
                        }
            },
            postaction={decorate}
}
]
  
  \filldraw (0,0) circle (2pt);
  \node [below] at (0,0) {$q_0$};
  
  \filldraw (10,0) circle (2pt);
  \node [below] at (10,0) {$q_1$};

  \filldraw (3.6,2.4) circle (2pt);
  \node [right] at (3.6,2.4) {$\Phi(s,q_t)$};

  \draw [nodewitharrow={0.6}{$t$}, thick] (0,0) to [out=20,in=-180] (5,.8) to[out=0,in=160] (10,0);
  \draw [nodewitharrow={0.3}{},thick] (0,0) to[out=60,in=-180] (5,5) to [out=0,in=130] (7,3) to [out=-50,in=130] (10,0);
  \draw [nodewitharrow={0.8}{\rotatebox{-90}{$s$}}, thick, dashed] (3.6,.752) to (3.6,4.73);
  
  \filldraw (3.6,4.73) circle (2pt);
  \node [above] at (3.6,4.73) {$\tilde q_t$};
  
  \filldraw (3.6,.752) circle (2pt);
  \node [below] at (3.6,.7) {$q_t$};
  \end{tikzpicture}
  \end{center}
  \caption{The perturbed curve}
  \label{fig:qtilde}
  \end{figure}
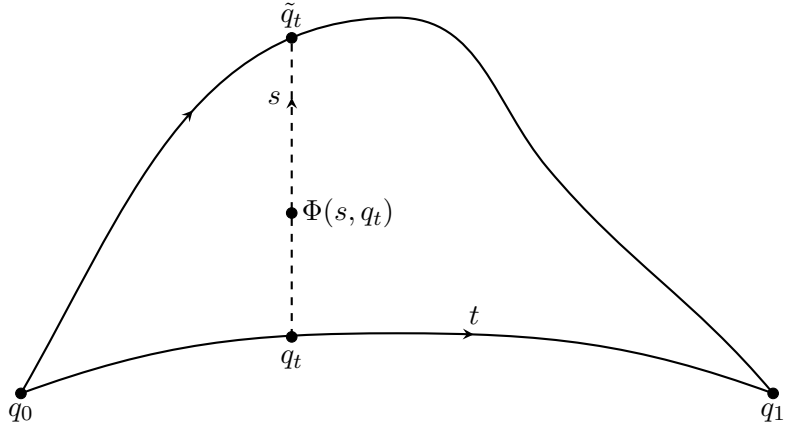
Note that the endpoints remain invariant, $\tilde q_0=q_0$ and $\tilde q_1=q_1$.
Since by definition of the flow $\partial_s \Phi(s,q_t)=-\nabla V(\Phi(s,q_t))$, the speed of the perturbed curve can be computed as
\[
\begin{split}
\frac{\d \tilde q_t}{\d t} & = \ddt\Big(\Phi(h(t),q_t)\Big) = \partial_s\Phi(h(t),q_t)h'(t) + \partial_q \Phi(h(t),q_t)\frac{\d q_t}{\d t} \\
& = -h'(t)\nabla V(\tilde q_t) +\partial_q \Phi(h(t),q_t)\frac{\d q_t}{\d t}.
\end{split}
\]
Bringing the $h'(t)$ term to the left-hand side and taking the half squared norm (in the tangent space $T_{\tilde q_t}M$) gives
\begin{equation}
 \label{eq:Riemannian_computation_pre}
 \frac{1}{2}\left|\frac{\d \tilde q_t}{\d t}\right|^2 +\frac{1}{2} |h'(t)|^2|\nabla V(\tilde q_t)|^2 
 +  h'(t) \underbrace{\langle \nabla V(\tilde q_t),\frac{\d\tilde q_t}{\d t}\rangle_{\tilde q_t}}_{=\ddt V(\tilde q_t)}
 = \frac{1}{2} \left| \partial_q \Phi(h(t),q_t)\frac{\d q_t}{\d t}\right|^2.
\end{equation}
Assume now that, for whatever reason, the gradient flow satisfies the following quantified contractivity estimate w.r.t.\ the Riemannian distance $d$
\begin{equation}
\label{hyp:lambda_contractivity}
d(\Phi(s,p_0),\Phi(s,p_0'))\leq e^{-\lambda s}d(p_0,p_0'),
\hspace{1cm}\forall s\geq 0,\, p_0,p_0'\in M
\end{equation}
for some fixed $\lambda\in \R$. Then it is easy to check that the linear map $v\mapsto \partial_q\Phi(s,p)\cdot v$ (from $T_pM$ to $T_{\Phi(s,p)}M$) has norm less than $e^{-\lambda s}$, and therefore \eqref{eq:Riemannian_computation_pre} gives
\begin{equation}
  \label{eq:Riemannian_computation}
 \frac{1}{2}\left|\frac{\d \tilde q_t}{\d t}\right|^2 +\frac{1}{2} |h'(t)|^2|\nabla V(\tilde q_t)|^2 
 +  h'(t) \ddt V(\tilde q_t)
 \leq   \frac{1}{2} e^{-2\lambda h(t)}\left| \frac{\d q_t}{\d t}\right|^2.
\end{equation}
Integration by parts yields next
\begin{equation}
\label{eq:fundamental}
\begin{split}
\frac{1}{2}\int_0^1\left|\frac{\d \tilde q_t}{\d t}\right|^2\,\dt & + \frac{1}{2} \int_0^1|h'(t)|^2|\nabla V(\tilde q_t)|^2 \,\dt - \int_0^1  h''(t) V(\tilde q_t)\,\dt \\
& \leq \frac{1}{2} \int_0^1 e^{-2\lambda h(t)}\left| \frac{\d q_t}{\d t}\right|^2\dt + \Big(h'(0)V(q_0)-h'(1)V(q_1)\Big),
\end{split}
\end{equation}
where the invariance $\tilde q_0=q_0$, $\tilde q_1=q_1$ was used in the last boundary terms.
This fundamental estimate gives a quantified bound on the kinetic energy (namely the $L^2$ speed) of $\tilde q$ in terms of that of the original curve $q$, and will be the cornerstone of the whole analysis.

Both the convexity and the convergence of the Schr\"odinger problem will actually follow by setting $h(t)=\eps H(t)$ for suitable choices of $H(t)\geq 0$, and then letting $\eps \downarrow 0$.
Note that in this case we have $h(t)=\eps H(t) \downarrow 0$ uniformly, hence the perturbed curve 
\begin{equation}\label{eq:perturbed-formal}
 q^\eps_t := \Phi(\eps H(t),q_t) 
\end{equation}
will converge uniformly to $q$ as $\eps \downarrow 0$ too.
%
\subsection{Convergence of the Schr\"odinger problem}

A first use of \eqref{eq:fundamental} will be crucial in proving the $\Gamma$-convergence of the Schr\"odinger functional
$$
\mathcal A_\eps(q):=\frac 12 \int_0^1 \left|\frac{\d q_t}{\d t}\right|^2\,\dt + \frac{\eps^2}{2}\int_0^1 |\nabla V(q_t)|^2\,\dt
$$
towards the kinetic action
$$
\mathcal A(q):=\frac 12 \int_0^1 \left|\frac{\d q_t}{\d t}\right|^2\,\dt
$$
as $\eps \downarrow 0$.
\begin{theo}[formal $\Gamma$-limit]
\label{theo:Gamma_lim_formal}
For any $q_0,q_1\in M$ it holds
$$
\mathcal A=\Gamma-\lim\limits_{\eps\to 0}\mathcal A_\eps
$$
for the uniform convergence on the space of curves with fixed endpoints $q_0,q_1$.
\end{theo}
\begin{proof}
We check separately the $\Gamma-\liminf$ and the $\Gamma-\limsup$ properties. As for the former, given any curve $q$ joining $q_0,q_1$ and any $q^\eps \to q$ uniformly, since the kinetic energy functional $q\mapsto \mathcal A(q)$ is always lower semicontinuous for the uniform convergence we get first
\[
\mathcal A(q) \leq \liminf_{\eps \downarrow 0} \mathcal A(q^\eps) \leq \liminf_{\eps \downarrow 0} \mathcal A_\eps(q^\eps).
\]
For the $\Gamma-\limsup$, let $H(t)=\min \{t,1-t\}$ be the hat function centered at $t=1/2$ with height $1/2$ and vanishing at the boundaries, set $h(t)=\eps H(t)$, and let $q^\eps$ be the regularized curve constructed in \eqref{eq:perturbed-formal}. In this simple smooth setting it is not difficult to check that $ q^\eps\to q$ uniformly. Moreover, our choice of $h(t)$ results in $|h'(t)|^2 = \eps^2$ with $h'(0)=\eps$, $h'(1)=-\eps$, and $h''(t)=-2\eps\delta_{1/2}(t)$ in the distributional sense. Therefore \eqref{eq:fundamental} gives immediately
\[
\begin{split}
\mathcal A_\eps(q^\eps) + 2\eps V(q^\eps_{1/2}) & = \frac 12\int_0^1\left|\frac{\d q^\eps_t}{\d t}\right|^2\,\dt + \frac{\eps^2}{2}\int_0^1 |\nabla V(q^\eps_t)|^2\,\dt + 2\eps V(q^\eps_{1/2})
\\
& \leq \frac 12 \int_0^1  e^{-2\eps \lambda H(t)}\left|\frac{\d q_t}{\d t}\right|^2\,\dt + \eps\Big(V(q_0)+V(q_1)  \Big).
\end{split}
\]
The singularity of $h''$ at $t=1/2$ can be easily and rigorously worked around, simply integrating by parts \eqref{eq:Riemannian_computation_pre} separately on each interval $t\in [0,1/2]$ and $t\in[1/2,1]$ and keeping track of the boundary terms resulting ultimately in the above $2\eps V(q^\eps_{1/2})\geq 0$ contribution.
Discarding this latter non-negative term finally gives
$$
\limsup_{\eps \downarrow 0} \mathcal A_\eps(q^\eps)
\leq 
\limsup_{\eps \downarrow 0} 
\left\{
\frac 12 \int_0^1  e^{-2\eps \lambda H(t)}\left|\frac{\d q_t}{\d t}\right|^2\,\dt + \eps\Big(V(q_0)+V(q_1)  \Big)\right\}
=
\mathcal A(q)
$$
and concludes the proof.
\end{proof}

\subsection{Quantifying the convexity}\label{sec:convexity-heuristic}

The second consequence of our fundamental estimate \eqref{eq:fundamental} is the quantification of the convexity of the potential $V$ in terms of the quantified contractivity \eqref{hyp:lambda_contractivity}.
The point here is that the result can be obtained directly from \eqref{eq:Riemannian_computation_pre}, which can be established in a purely metric setting without relying on differential calculus (see the next section for details).

\begin{theo}
\label{theo:V_convex}
Assume that $V$ satisfies \ref{hyp:lambda_contractivity}.
Then $V$ is $\lambda$-geodesically convex, i.e.
$$
V(q_\theta)\leq (1-\theta)V(q_0)+ \theta V(q_1) - \frac{\lambda}{2}\theta(1-\theta) d^2(q_0,q_1),
\qquad \theta\in(0,1)
$$
for \emph{any} geodesic $(q_\theta)_{\theta\in[0,1]}$ in $M$.
\end{theo}

\begin{proof}
Let $(q_t)_{t\in[0,1]}$ be an arbitrary geodesic with endpoints $q_0,q_1$.
For fixed $\theta\in(0,1)$ let
$$
 H_\theta(t):=
\begin{cases}
  \displaystyle{\frac{1}{\theta}} t & \mbox{if }t\in [0,\theta],\\
  - \displaystyle{\frac{1}{1-\theta}(t-1)} & \mbox{if }t\in [\theta,1],
\end{cases}
 $$
be the hat function centered at $t=\theta$ with height $1$ and vanishing at $t=0,1$, and for any $\eps > 0$ let $q^\eps$ be the regularized curve constructed in \eqref{eq:perturbed-formal} with $h(t)=\eps H_\theta(t)$. Note moreover that 
\[
h'(0)=\frac \eps\theta, \qquad h'(1)=-\frac\eps{1-\theta}, \qquad h''(t)=-\eps\left(\frac{1}{\theta}+\frac 1{1-\theta}\right)\delta_{\theta}(t)
\]
in the distributional sense. Discarding the non-negative term $|h'(t)|^2|\nabla V(\tilde q_t)|^2$ in \eqref{eq:fundamental}, the optimality of the geodesic $q$ from $q_0$ to $q_1$ gives
\[
\begin{split}
0 & \leq \frac 12 \int_0^1 \left|\frac{\d q^\eps_t}{\d t}\right|^2 \,\dt - \frac 12 \int_0^1 \left|\frac{\d q_t}{\d t}\right|^2 \,\dt 
\\
& \overset{\eqref{eq:fundamental}}{\leq} \int_0^1  h''(t) V( q^\eps_t)\,\dt
+   \frac{1}{2} \int_0^1 \left(e^{-2\lambda h(t)}-1\right)\left| \frac{\d q_t}{\d t}\right|^2\,\dt 
 + \Big(h'(0)V(q_0)-h'(1)V(q_1)\Big)
 \\
& = -\eps\left(\frac 1\theta+\frac 1{1-\theta}\right) V(q^\eps_\theta) + \frac{d^2(q_0,q_1)}{2}\int_0^1\left(e^{-2\lambda \eps H_\theta(t)}-1\right) \dt \\
& \qquad +\eps\left(\frac 1\theta V(q_0)+\frac 1{1-\theta}V(q_1)\right),
\end{split}
\]
where the last equality follows from the constant speed property $|\frac{\d q_t}{\d t}|^2 = d^2(q_0,q_1)$ of the geodesic $(q_t)_{t\in[0,1]}$ connecting $q_0,q_1$ as well as from the explicit properties of $h(t)=\eps H_\theta(t)$ listed above.
Multiplying by $\frac{\theta(1-\theta)}{\eps}>0$ and rearranging gives
 $$
 V(q^ \eps_\theta)\leq (1-\theta)V(q_0)+\theta V(q_1)
 +\theta(1-\theta)\frac{d^2(q_0,q_1)}{2} \underbrace{\int_0^1\frac{e^{-2\lambda \eps H_\theta(t)}-1}{\eps} \dt}_{:=I_\eps}.
 $$
Since $\int_0^1H_{\theta}(t)\d t=\frac 12$ for all $\theta$ we see that $I_\eps\to -2\lambda\int_0^1 H_\theta(t)\,\dt = -\lambda$ as $\eps \downarrow 0$, and the result immediately follows since $V(q^\eps_\theta)\to V(q_\theta)$ as well in the left-hand side.
\end{proof}
%
%

\section{Estimates in metric spaces}
\label{sec:metric spaces}

Before trying to adapt the previous computations to the metric context we need to fix once and for all the framework to be used in the sequel.

\subsection{Preliminaries and setting}\label{sec:preliminaries}
\begin{itemize}
 \item 
By $C([0,1],(\X,\sfd))$, or simply $C([0,1],\X)$, we denote the space of continuous curves with values in the metric space $(\X,\sfd)$.
The collection of absolutely continuous curves on $[0,1]$ is denoted by $AC([0,1],(\X,\sfd))$, or simply by $AC([0,1],\X)$.
For any curve $(\gamma_t) \in AC([0,1],\X)$, its length is well defined as
\[
\ell(\gamma) := \int_0^1 |\dot{\gamma}_t|\,\d t,
\]
where $|\dot{\gamma}_t|$ denotes the metric speed of $\gamma$.
If $|\dot{\gamma}_t| \in L^2(0,1)$, then we shall say that $(\g_t) \in AC^2([0,1],\X)$.
For these notions of absolutely continuous curves and metric speed in a metric space, see for instance \cite[Section 1.1]{AmbrosioGigliSavare08}.
\item
A curve $\gamma : [0,1] \to \X$ is called geodesic provided $\sfd(\gamma_t,\gamma_s) = |t-s|\sfd(\gamma_0,\gamma_1)$ for all $t,s \in [0,1]$.
\item
The slope $|\partial\V|$ of a functional $\V : \X \to \R \cup \{+\infty\}$ at a point $x \in \X$ is set as $+\infty$ if $x \notin D(\V)$, $0$ if $x$ is isolated, and defined as
\[
|\partial \V|(x) := \limsup_{y \to x}\frac{[\V(x) - \V(y)]^+}{\sfd(x,y)}
\]
if $x \in D(\V)$.
\item
A curve $(\gamma_t)_{t>0} \subset \X$ is said to be a gradient flow of $\V$ in the $\EVI_\lambda$ sense (with $\lambda \in \R$) provided $(\gamma_t) \in AC_{loc}((0,\infty),\X)$ and
\begin{equation}
\label{eq:EVIlambda}
\tag{$\EVI_\lambda$}
\frac{1}{2}\ddt \sfd^2(\gamma_t,y) + \frac{\lambda}{2}\sfd^2(\gamma_t,y) + \V(\gamma_t) \leq \V(y), \qquad \forall y \in \X,\,\textrm{a.e. } t > 0.
\end{equation}
If $\gamma_t \to x$ as $t \downarrow 0$ with $x \in \overline{D(\V)}$, then we say that the gradient flow $(\gamma_t)$ starts at $x$.
\end{itemize}
\bigskip

After this premise, let us fix the framework we shall work within.

\begin{setting}
\label{setting}
On the space $(\X,\sfd)$ and on the functional $\V : \X \to \mathbb{R} \cup \{+\infty\}$ we make the following assumptions:
\begin{enumerate}[label={(A\arabic*})]
\item
\label{item:complete_sep}
$(\X,\sfd)$ is a complete and separable metric space;
\item
\label{item:V_locally_bounded_below}
$\V$ is lower semicontinuous with dense domain, i.e.\ $\overline{D(\V)} = \X$, and locally bounded from below in the following sense: for any $\sfd$-bounded set $B \subset \X$ there exists $c_B \in \mathbb{R}$ such that $\V(x) \geq c_B$ for all $x \in B$;
\item
\label{item:generate_lambda_flow}
there exists $\lambda \in \mathbb{R}$ such that for any $x \in \X$ there exists an $\EVI_\lambda$-gradient flow of $\V$ starting from $x$.
In view of \eqref{eq:contraction}, the corresponding 1-parameter semigroup shall be denoted $\S_t$.
\end{enumerate}
\end{setting}

Sometimes, and always explicitly indicated, we will also use the following extra hypothesis. 

\begin{hyp} \label{hyp1}
There exists a Hausdorff topology $\sigma$ on $\X$ such that $\sfd$-bounded sets are sequentially $\sigma$-compact.
Moreover, the distance $\sfd$ and the slope $|\partial\V|$ are $\sigma$-sequentially lower semicontinuous. 
\end{hyp}
\begin{rmk}\label{rmk18} {\rm
Assumption \ref{hyp1} is in particular valid provided $(\X,\sfd)$ is a locally compact space. 
Indeed, in this case the metric topology of $(\X,\sfd)$ is an admissible candidate for $\sigma$, since bounded sets are relatively compact (by \cite[Proposition 2.5.22]{BBI01}) and the lower semicontinuity of the slope $|\partial\V|$ w.r.t.\ the metric topology is a consequence of the forthcoming identity \eqref{eq:local=global}.
}\fr 
\end{rmk}
\begin{rmk}\label{rmk19} {\rm 
Assumption \ref{hyp1} implies that $\sfd$-converging sequences are also $\sigma$-converging.
Indeed, given $(x_n) \subset \X$ with $\sfd(x,x_n) \to 0$ as $n \to \infty$ for some limit $x \in \X$, by Assumption \ref{hyp1} and by the boundedness of $(x_n)_n$ there exist a subsequence $(x_{n_k})_k$ and $y \in \X$ such that
\[
x_{n_k} \stackrel{\sigma}{\to} y \quad \textrm{as}\quad k \to \infty.
\]
Since $\sfd$ is $\sigma$-sequentially lower semicontinuous (again by Assumption \ref{hyp1}) we deduce that
\[
\sfd(x,y) \leq \liminf_{k \to \infty}\sfd(x,x_{n_k}) = \lim_{n \to \infty}\sfd(x,x_n) = 0,
\]
whence $x=y$.
This classically implies that the whole sequence converges, $x_{n} \stackrel{\sigma}{\to} x$.
}\fr 
\end{rmk}
We list now some useful properties of $\EVI$-gradient flows, which hold true in Setting \ref{setting} and that we shall use extensively in the sequel.
First of all, whenever $x\in X$ is the starting point of an $\EVI_\lambda$ flow, the slope there (a \emph{local} object, a priori) admits the \emph{global} representation
\begin{equation}
|\partial\V|(x) = \sup_{y \neq x}\Big(\frac{\V(x) - \V(y)}{\sfd(x,y)} + \frac{\lambda}{2}\sfd(x,y)\Big)^+ ,
\label{eq:local=global}
\end{equation}
see \cite[Proposition 3.6]{MuratoriSavare20}.
Since we assume that any $x\in X$ is the starting point of an $\EVI_\lambda$-gradient flow, this means in particular that $|\partial\V|:X \to [0,\infty]$ is lower semicontinuous, since so is the right-hand side above (as a supremum of lower semicontinuous functions).
This also implies by \cite[Theorem 1.2.5]{AmbrosioGigliSavare08} that $|\partial\V|$ is a \emph{strong upper gradient} for $\V$ in the sense of \cite[Definition 1.2.1]{AmbrosioGigliSavare08}, namely: for every $(\g_t) \in AC([0,1],\X)$, the map $t \mapsto \V(\g_t)$ is Borel and
\begin{equation}
\label{eq:strong upper gradient}
|\V(\g_{t_1}) - \V(\g_{t_0})| \leq \int_{t_0}^{t_1} |\partial\V|(\g_t)|\dot \g_t|\,\d t, \qquad \forall 0 \leq t_0 \leq t_1 \leq 1,
\end{equation}
the right-hand side being possibly infinite.
In addition, if $(\gamma_t)$ is an $\EVI_\lambda$-gradient flow of $\V$ then the following hold \cite[Theorem 3.5]{MuratoriSavare20}:
\begin{enumerate}[label={(\roman*)}]
\item
If $(\gamma_t)$ starts from $x \in \overline{D(\V)}$ and $(\tilde{\gamma}_t)$ is a second $\EVI_\lambda$-gradient flow of $\V$ starting from $y \in \overline{D(\V)}$ respectively, then
\begin{equation}
\label{eq:contraction}
\sfd^2(\gamma_t,\tilde{\gamma}_t) \leq e^{-2\lambda t}\sfd^2(x,y), \qquad \forall t \geq 0.
\end{equation}
This means that EVI-gradient flows are unique (provided they exist) and thus if there exists an EVI-gradient flow $(\gamma_t)$ starting from $x$, then a 1-parameter semigroup $(\S_t)_{t>0}$ is unambiguously associated to it via $\S_tx = \gamma_t$.
\item
The maps $t \mapsto \gamma_t$ and $t \mapsto \V(\gamma_t)$ are locally Lipschitz in $(0,\infty)$ with values in $\X$ and $\R$, respectively, and satisfy the \emph{Energy Dissipation Equality}
\begin{equation}\label{eq:speed=slope}
-\ddt \V(\gamma_t) = \frac{1}{2}|\dot{\gamma}_t|^2 + \frac{1}{2}|\partial \V|^2(\gamma_t) = |\dot{\gamma}_t|^2 = |\partial \V|^2(\gamma_t), \qquad \textrm{for a.e. } t>0.
\end{equation}
\item
The map
\begin{equation}\label{eq:monotonicity-slope}
t \mapsto e^{\lambda t}|\partial \V|(\gamma_t) \quad \textrm{is non-increasing}.
\end{equation}
\item
If $(\gamma_t)$ starts from $x$ and $y \in D(|\partial \V|)$, then
\begin{equation}\label{eq:regularization}
|\partial \V|^2(\gamma_t) \leq \frac{1}{2e^{\lambda t}-1}|\partial \V|^2(y) + \frac{1}{I_\lambda(t)^2}\sfd^2(x,y), \qquad \textrm{provided } -\lambda t < \log 2,
\end{equation}
where $I_\lambda(t) := \int_0^t e^{\lambda s}\,\d s$.
\end{enumerate}
We emphasize that these properties directly follow from the very definition \eqref{eq:EVIlambda} of gradient flows, and a priori do not require $\V$ to be geodesically $\lambda$-convex.
Although analogous statements can be found in \cite{AmbrosioGigliSavare08} and \cite{AmbrosioGigli11} under convexity assumptions on $\V$, the latter are essentially needed to grant \emph{existence} of EVI-gradient flows.
It is important to stress this fact because in Setting \ref{setting} we \emph{assume} that for any $x \in \X$ there exists an $\EVI_\lambda$-gradient flow of $\V$ starting there, which from \cite{DaneriSavare08} is known to imply that $\V$ is geodesically convex.
In Section \ref{sec:metric spaces} we also provide an alternative proof of this latter fact, whence the necessity for us to avoid all properties of $\EVI$-gradient flows actually relying on geodesic convexity.

We conclude this preliminary part with a general integrability result about $\EVI_\lambda$-gradient flows, which we could not find explicitly written in the literature and will be used later on in the proof of Lemma~\ref{lem:technical formula}.

\begin{lem}\label{lem:integrability}
With the same assumptions and notations as in Setting \ref{setting}, let $x \in \X$. Then
\begin{center}
$t \mapsto \V(\S_t x)$ is integrable in $[0,T]$, for all $T>0$,
\end{center}
regardless of whether $\V(x)$ is initially finite or not.
\end{lem}
On intervals $[\eps,T]$ this computation is easily justified by the fact that $t \mapsto \V(\S_t x)$ is locally Lipschitz in $(0,\infty)$, hence locally integrable therein.
But this computation is legitimate even if $\eps=0$, as we are going to see.

\begin{proof}
Let $x \in \X$ and $T>0$ be as in our statement.
Since $\V$ is bounded from below on $\sfd$-bounded sets by \ref{item:V_locally_bounded_below}, and because $(\S_t x)_{t \in [0,T]}$ is bounded, there exists $c \in \mathbb{R}$ such that $\V(\S_t x) \geq c$ for all $t \in [0,T]$.
Combining with \eqref{eq:EVIlambda} this gives
\[
c \leq \V(\S_t x) \leq \V(y) - \frac{1}{2}\frac{\d}{\d t}\sfd^2(\S_t x,y) - \frac{\lambda}{2}\sfd^2(\S_t x,y)
\]
for any $y \in D(\V)$ and $t \in (0,T]$.
Integrating from $t=\eta>0$ to $t=T$ gives
\[
\begin{split}
c(T-\eta) \leq \int_\eta^T \V(\S_t x)\,\d t & \leq (T-\eta)\V(y) - \frac 12 \Big(\sfd^2(\S_T x,y) - \sfd^2(\S_\eta x,y)\Big) \\
& \qquad - \frac{\lambda}{2}\int_\eta^T \sfd^2(\S_t x,y)\,\d t.
\end{split}
\]
As $t \mapsto \V(\S_t x)$ is bounded from below on $[0,T]$ and the right-hand side has a finite limit as $\eta \downarrow 0$ (thanks to the fact that $t \mapsto \S_t x$ is $\sfd$-continuous on $[0,\infty)$ by the very definition of $\EVI_\lambda$-gradient flow), we deduce the desired integrability.
\end{proof}
\subsection{A pseudo-Riemannian computation}

In this section the formal Riemannian computations carried out at the beginning of Section \ref{sec:heuristics}, and more precisely \eqref{eq:Riemannian_computation}, will be reproduced rigorously in the abstract Setting \ref{setting}.
To this aim, a key role will be played by the following purely metric estimate:

\begin{lem}\label{lem:technical formula}
With the same assumptions and notations as in Setting \ref{setting}, let $(\gamma_t) \in AC([0,1],\X)$ with $\V(\g_0),\V(\g_1) < \infty$.
For any fixed absolutely continuous function $h : [0,1] \to \R$ with $h(t)>0$ for all $t \in (0,1)$ let
\[
\tilde \g_t:=\S_{h(t)}\g_t, \qquad t \in [0,1],
\]
and for any $0\leq t_0< t_1\leq 1$ write
\begin{equation}
\label{eq:def_tdelta}
t^+:=\left\{
\begin{array}{ll}
 t_1 & \mbox{if }h(t_1)\geq h(t_0)\\
 t_0 & \mbox{otherwise}
\end{array}
\right.
\qqtext{and}
t^-:=\left\{
\begin{array}{ll}
 t_0 & \mbox{if }h(t_1)\geq h(t_0)\\
 t_1 & \mbox{otherwise}
\end{array}
\right. .
\end{equation}
Then we have the exact estimate
\begin{equation}
\label{eq:pseudo_Riemannian_estimate_discrete}
\begin{split}
\frac{1}{2}\left|\frac{\sfd(\tilde \g_{t_1},\tilde \g_{t_0})}{t_1-t_0}\right|^2
& + 
\frac{1}{2\lambda^2} |\partial \V|^2(\tilde \g_{t^+})\frac{e^{\lambda(h(t_1)-h(t_0))} + e^{\lambda(h(t_0)-h(t_1))} - 2}{(t_1-t_0)^2} \\
& +
\frac{1-e^{-\lambda (h(t^+)-h(t^-))}}{\lambda(t^+ - t^-)} \cdot\frac{\V(\tilde \g_{t_1})-\V(\tilde \g_{t_0})}{t_1-t_0}\\
& \qquad \leq 
\frac{1}{2} e^{-\lambda (h(t_1)+h(t_0))}\left|\frac{\sfd(\g_{t_1},\g_{t_0})}{t_1-t_0}\right|^2.
\end{split}
\end{equation}
\end{lem}
\noindent
Here we use the convention that $(+\infty)\times 0=0$ whenever $|\partial \V|(\tilde\g_{t^+})=+\infty$ and $h(t_0)=h(t_1)$ in the second term on the left-hand side of \eqref{eq:pseudo_Riemannian_estimate_discrete}.
Since we assume that $h(t)>0$ for $t\in(0,1)$, and because any $\EVI_\lambda$-gradient flow immediately falls within $D(|\partial\V|)$ by standard regularizing effects, this latter case is in fact only possible if $t_0=0$, $t_1=1$, and $h(t_0)=h(t_1)=0$.
In that case $\tilde\gamma_{0}=\gamma_0$ and $\tilde\gamma_1=\gamma_1$, the third term in the left-hand side also cancels owing to $e^{-\lambda (h(t^+)-h(t^-))}=1$, and \eqref{eq:pseudo_Riemannian_estimate_discrete} then holds as a trivial equality.

We shall rely on this lemma later on in two different ways:
First, fixing $t_0=0$ and letting $t_1\downarrow 0$ (resp. fixing $t_1=1$ and letting $t_0\uparrow 1$) to control in Lemma~\ref{lem:Vqtilda_is_AC} the continuity of $t\mapsto \V(\tilde\g_{t})$ at the boundaries $t=0,1$, and second, fixing $t_0\in(0,1)$ and letting $t_1\to t_0$ to obtain in Proposition~\ref{prop:pseudo_Riemannian_estimate} a pointwise differential estimate similar to \eqref{eq:Riemannian_computation}.

\begin{rmk}\label{rmk:fancy}{\rm
The times $t^\pm$ are just a convenient notation, ordered as $h(t^-)\leq h(t^+)$. Note that in our estimate \eqref{eq:pseudo_Riemannian_estimate_discrete} the Fisher information $|\partial \V|^2(\tilde \g_{t^+})$ is evaluated at the time $t=t^+$ for which the ``smoothing time'' $s=h(t_0)$ or $s=h(t_1)$ is the largest, i.e.\ where the regularizing vertical flow has been run for the longest time.
This is somehow natural, as this specific point is ``better'' than the other one in terms of regularity.
}\fr 
\end{rmk}

\begin{proof}
By symmetry we only discuss the case $h(t_1) \geq h(t_0)$, i.e.\ $t^+ = t_1$ and $t^- = t_0$.
As already mentioned, if $h(t_0)=h(t_1)=0$ our statement is actually vacuous, thus it is not restrictive to further assume $h(t_1)>0$.
Let us write for simplicity
$$
\hat \g_{t_1}:= \S_{h(t_0)}\g_{t_1}.
$$
From an intuitive point of view, this corresponds to freezing a ``vertical time'' $s=h(t_0)$ and ``translating'' $\tilde \g_{t_0}$ in the ``horizontal'' $t$ direction parallel to the curve $\g$ until $t_1$. Here, in the ``vertical'' direction above $t_1$ the smoothing semigroup $\S_s$ associated with $\V$ has been run at least for a strictly positive time $h(t_1)>0$, so that by \eqref{eq:regularization} the solution of the ``vertical'' gradient flow at that time lies within the regular domain $\X_1 = D(|\partial \V|) \subset \X_0 = D(\V) \subset \X$, see Figure~\ref{fig:horizontal-vertical}.

  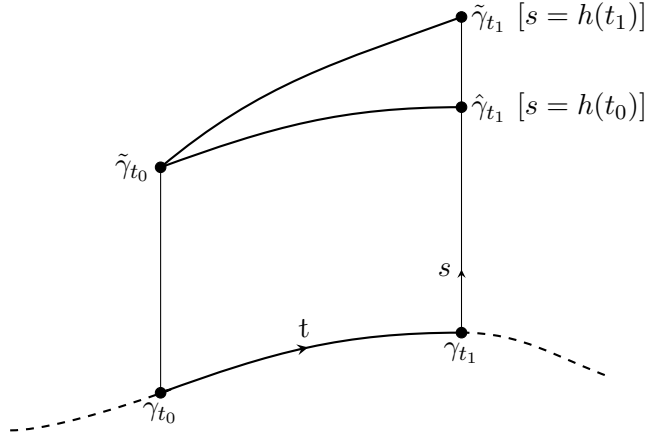
\begin{figure}[h!]
  \begin{center}
  \begin{tikzpicture}
  [
nodewitharrow/.style 2 args={                
            decoration={             
                        markings,   
                        mark=at position {#1} with { 
                                    \arrow{stealth},
                                    \node[transform shape,above] {#2};
                        }
            },
            postaction={decorate}
}
]
  
  \filldraw (3,1) circle (2pt);
  \node [below] at (3,1) {$\g_{t_0}$};
  
  \filldraw (7,1.8) circle (2pt);
  \node [below] at (7,1.8) {$\g_{t_1}$};
  
  \filldraw (3,4) circle (2pt);
  \node [left] at (3,4) {$\tilde \g_{t_0}$};
  
  \filldraw (7,4.8) circle (2pt);
   \node [right] at (7,4.8) {$\hat \g_{t_1}\,\,[s=h(t_0)]$};
   
   \filldraw (7,6) circle (2pt);
   \node [right] at (7,6) {$\tilde \g_{t_1}\,\,[s=h(t_1)]$};
  \draw [dashed,thick] (1,.5) to [out=0,in=20] (3,1);
  \draw [nodewitharrow={0.5}{$t$}, thick] (3,1) to [out=20,in=-180] (7,1.8);
  \draw [ thick] (3,4) to [out=20,in=-180] (7,4.8);
  \draw [dashed,thick] (7,1.8) to [out=0, in=160] (9,1.2);
  \draw [thick] (3,4) to [out=40,in=-160] (7,6);
  \draw  [] (3,1) to (3,4);
  \draw  [nodewitharrow={0.2}{\rotatebox{-90}{$s$}}] (7,1.8) to (7,6);
  \end{tikzpicture}
  \end{center}
  \caption{The horizontal and vertical curves}
\label{fig:horizontal-vertical}
  \end{figure}

The first step is to write \eqref{eq:EVIlambda} for $s\mapsto \S_s(\g_{t_1})$ with $\tilde \g _{t_0}$ as a reference point, namely
\[
\frac 12 \frac{\d}{\d s}\sfd^2(\S_s \g_{t_1},\tilde \g_{t_0}) + \frac{\lambda}{2}\sfd^2(\S_s \g_{t_1},\tilde \g_{t_0}) + \V(\S_s \g_{t_1}) \leq \V(\tilde \g_{t_0}),
\]
which holds true for a.e.\ $s \in [0,h(t_1)]$ in the ``vertical'' direction. This inequality can be equivalently rewritten as
\begin{equation}\label{eq:first step}
\frac 12 \frac{\d}{\d s}\Big(e^{\lambda s}\sfd^2(\S_s \g_{t_1},\tilde \g_{t_0})\Big) \leq e^{\lambda s}\Big(\V(\tilde \g_{t_0})-\V(\S_s \g_{t_1})\Big).
\end{equation}
Note that this estimate carries significant information if and only if the reference point has finite entropy, i.e.\ $\V(\tilde \g_{t_0})<\infty$ in the right-hand side.
This holds true for $t_0\in(0,1)$ because $\tilde \g_{t_0}$ is the $\EVI_\lambda$-gradient flow of $\V$ starting from $\g_{t_0}$ at a strictly positive time $s=h(t_0)>0$, but also for $h(t_0)=0$ if $t_0=0$ since in this case $\tilde \g_0=\g_0$ is assumed to have finite entropy.

Integrating \eqref{eq:first step} from $s=h(t_0)$ to $s=h(t_1)$ gives
\begin{equation}
\label{eq:legitimate}
\begin{split}
\frac{1}{2}e^{\lambda h(t_1)}\sfd^2(\tilde \g_{t_1},\tilde \g_{t_0}) & - \frac{1}{2}e^{\lambda h(t_0)}\sfd^2(\hat \g_{t_1},\tilde \g_{t_0}) \\
& \leq \int_{h(t_0)}^{h(t_1)} e^{\lambda s}\Big(\V(\tilde \g_{t_0}) - \V(\S_s\g_{t_1})\Big) \,\d s \\
& = \int_{h(t_0)}^{h(t_1)} e^{\lambda s}\Big(\big(\V(\tilde \g_{t_0})- \V(\tilde \g_{t_1})\big) + \big(\V(\tilde \g_{t_1}) - \V(\S_s\g_{t_1})\big)\Big) \,\d s \\
& = \int_{h(t_0)}^{h(t_1)} e^{\lambda s}\Big(\V(\tilde \g_{t_1})-\V(\S_s\g_{t_1})\Big) \,\d s \\
& \qquad - \frac{e^{\lambda h(t_1)}-e^{\lambda h(t_0)}}{\lambda} \Big(\V(\tilde \g_{t_1})- \V(\tilde \g_{t_0})\Big).
\end{split}
\end{equation}
If $h(t_0)>0$ this computation is legitimate because $s \mapsto \S_s \g_{t_1}$ and $s \mapsto \V(\S_s \g_{t_1})$ are locally Lipschitz in $(0,\infty)$, hence $s \mapsto \sfd^2(\S_s \g_{t_1},\tilde \g_{t_0})$ and $s \mapsto \V(\S_s \g_{t_1})$ are locally integrable therein.
But this computation is also justified when $h(t_0)=0$ by Lemma \ref{lem:integrability}.
More specifically $s \mapsto \V(\S_s\g_{t_1})$ is absolutely integrable on $[0,T]$ for any $T>0$ and a fortiori so is $s \mapsto e^{\lambda s}\V(\S_s\g_{t_1})$.

Now let us estimate the terms in \eqref{eq:legitimate} to get \eqref{eq:pseudo_Riemannian_estimate_discrete}.
First, since $\hat \g_{t_1}=\S_{h(t_0)}\g_{t_1}$, $\tilde \g_{t_0} = \S_{h(t_0)}\g_{t_0}$, and $\S_{h(t_0)}(\cdot)$ is $\lambda$-contractive by \eqref{eq:contraction}, we observe that the second term in the left-hand side of \eqref{eq:legitimate} can be controlled as
\begin{equation}\label{eq:easy term}
\sfd^2(\hat \g_{t_1},\tilde \g_{t_0}) = \sfd^2(\S_{h(t_0)}\g_{t_1},\S_{h(t_0)}\g_{t_0}) \leq e^{-2\lambda h(t_0)} \sfd^2(\g_{t_1},\g_{t_0}).
\end{equation}
On the right-hand side, let us define
\[
I := \int_{h(t_0)}^{h(t_1)} e^{\lambda s}\Big(\V(\tilde \g_{t_1})-\V(\S_s\g_{t_1})\Big) \,\d s.
\]
This integral is clearly non-positive by \eqref{eq:speed=slope}, but we need a finer analysis. To this aim, for fixed $0<s<h(t_1)$ let us write
\[
\begin{split}
\V(\tilde \g_{t_1}) - \V(\S_s\g_{t_1}) & = \V(\S_{h(t_1)}\g_{t_1}) - \V(\S_s\g_{t_1}) \\
& = \int_{s}^{h(t_1)}\frac{\d}{\d\tau} \V(\S_{\tau}\g_{t_1})\,\d\tau = -\int _{s}^{h(t_1)} |\partial \V|^2(\S_{\tau}\g_{t_1})\,\d\tau,
\end{split}
\]
where the second equality holds due to $\tau\mapsto \V(\S_\tau \g_{t_1})$ being Lipschitz on $[s,h(t_1)]$, and the third one stems from \eqref{eq:speed=slope} for the gradient flow $\tau\mapsto \S_\tau \g_{t_1}$.
By \eqref{eq:monotonicity-slope}
\[
\begin{split}
-\int _{s}^{h(t_1)} |\partial \V|^2(\S_{\tau}\g_{t_1})\,\d\tau & \leq -\int_s^{h(t_1)} |\partial\V|^2(\S_{h(t_1)}\g_{t_1}) e^{2\lambda(h(t_1)-\tau)}\,\d\tau \\
& = - e^{2\lambda h(t_1)}|\partial\V|^2(\S_{h(t_1)}\g_{t_1}) \int_s^{h(t_1)} e^{-2\lambda\tau}\,\d\tau \\
& = \frac{1}{2\lambda}\Big(1 - e^{2\lambda(h(t_1)-s)}\Big) |\partial\V|^2(\S_{h(t_1)}\g_{t_1}),
\end{split}
\]
so that, as a consequence,
\[
\begin{split}
I & \leq \frac{1}{2\lambda}\int_{h(t_0)}^{h(t_1)}e^{\lambda s}\Big(1 - e^{2\lambda (h(t_1)-s)}\Big) |\partial\V|^2(\S_{h(t_1)}\g_{t_1})\,\d s \\
& = \frac{1}{2\lambda}|\partial\V|^2(\S_{h(t_1)}\g_{t_1})\int_{h(t_0)}^{h(t_1)}\Big(e^{\lambda s} - e^{2\lambda h(t_1)}\cdot e^{-\lambda s}\Big)\,\d s \\
& = -\frac{1}{2\lambda^2}|\partial\V|^2(\tilde \g_{t_1})e^{\lambda h(t_1)}\Big(e^{\lambda(h(t_1)-h(t_0))} + e^{\lambda(h(t_0)-h(t_1))} - 2\Big).
\end{split}
\]
Plugging this estimate together with \eqref{eq:easy term} into \eqref{eq:legitimate} and dividing by $(t_1-t_0)^2>0$ entails our claim.
\end{proof}
We also need to study the behaviour of the ``regularized'' curve $\tilde \g_t := \S_{h(t)}\g_t$ and of the entropy $\V$ along it: this is the content of the following two results.

\begin{lem}
\label{lem:qtilda_is_AC}
With the same assumptions and notations as in Setting \ref{setting}, if $(\g_t) \in AC([0,1],\X)$ and $h : [0,1] \to \R$ is absolutely continuous with $h(t)>0$ for all $t \in (0,1)$, then the curve $\tilde \g_t := \S_{h(t)}\g_t$ belongs to $AC_{loc}((0,1),\X) \cap C([0,1],\X)$.
\end{lem}

\begin{proof}
Fix $\delta \in (0,1/2)$, $t_0, t_1 \in [\delta,1-\delta]$ with $t_0 \leq t_1$ and define
\begin{equation}\label{eq:min max}
m_\delta := \min_{t \in [\delta,1-\delta]}h(t), \qquad M_\delta := \max_{t \in [\delta,1-\delta]} h(t),
\end{equation}
paying attention to the fact that $m_\delta > 0$ by construction.
Write as before $\hat \g_{t_1}:= \S_{h(t_0)}\g_{t_1}$ for the ``horizontal'' translation of $\tilde \g_{t_0}$ (see Figure~\ref{fig:horizontal-vertical}).
By triangular inequality and the contraction estimate \eqref{eq:contraction} we get
\begin{equation}\label{eq:triangular}
\sfd(\tilde \g_{t_0},\tilde \g_{t_1}) 
\leq 
\sfd(\tilde \g_{t_0},\hat \g_{t_1}) 
+
\sfd(\hat \g_{t_1},\tilde \g_{t_1})
\leq 
e^{\lambda^- M_\delta}\sfd(\g_{t_0},\g_{t_1}) 
+
\sfd(\hat \g_{t_1},\tilde \g_{t_1}),
\end{equation}
where $\lambda^- := \max\{-\lambda,0\}$. Since $(\g_t)$ is absolutely continuous the first term in the right-hand side can be controlled as $\sfd( \g_{t_0},\g_{t_1}) \leq \int_{t_0}^{t_1}|\dot \g_t|\d t$.
As regards the second one, by \eqref{eq:speed=slope} and up to assuming $h(t_0) \leq h(t_1)$ (which is not restrictive, as otherwise it is sufficient to swap the boundary values of integration below) it holds
\begin{equation}\label{eq:asimov}
\sfd(\hat \g_{t_1},\tilde \g_{t_1}) = \sfd(\S_{h(t_0)} \g_{t_1},\S_{h(t_1)} \g_{t_1}) \leq \int_{h(t_0)}^{h(t_1)}\Big|\frac{\d}{\d s}\S_s\g_{t_1}\Big|\,\d s = \int_{h(t_0)}^{h(t_1)}|\partial\V|(\S_s\g_{t_1})\,\d s,
\end{equation}
where, to avoid possibly ambiguous notations, $|\frac{\d}{\d s}\S_s\g_{t_1}|$ denotes the metric speed of the ``vertical'' curve $s \mapsto \S_s\g_{t_1}$.
In order to control the slope in the right-most term uniformly both in $s \in [h(t_0),h(t_1)] \subset [m_\delta,M_\delta]$ and in $t_1 \in [\delta,1-\delta]$, for fixed $\delta$, let $\eps$ be such that $-\lambda\eps < \log 2$ (if $\lambda \geq 0$, choose $\eps = m_\delta$) and define $\eps' := \min\{m_\delta,\eps\}$.
Then by \eqref{eq:monotonicity-slope} and the fact that $s \geq h(t_0) \geq m_\delta \geq \eps'$ we have
\[
|\partial\V|(\S_s\g_{t_1}) \leq e^{\lambda(\eps'-s)}|\partial\V|(\S_{\eps'}\g_{t_1}) \leq e^{\lambda^-(M_\delta - \eps')}|\partial\V|(\S_{\eps'}\g_{t_1}), \qquad \forall s \in [m_\delta,M_\delta]
\]
and by \eqref{eq:regularization} for any reference point $x \in D(|\partial\V|)$ it holds
\[
|\partial\V|^2(\S_{\eps'}\g_{t_1}) \leq \frac{1}{2e^{\lambda\eps'}-1}|\partial\V|^2(x) + \frac{1}{I_\lambda(\eps')^2}\sfd^2(x,\g_{t_1}).
\]
The squared distance in the right-hand side above is bounded uniformly in $t_1 \in [\delta,1-\delta]$, since by triangular inequality
\[
\sfd(x,\g_{t_1}) \leq \sfd(x,\g_0) + \sfd(\g_0,\g_{t_1}) \leq \sfd(x,\g_0) + \ell(\g).
\]
Therefore there exists $C_\delta > 0$ such that
\begin{equation}
\label{eq:estimate_slope_uniform}
|\partial\V|(\S_s\g_{t_1}) \leq C_\delta
\qqtext{for all}
t_1\in[\delta,1-\delta]
\mbox{ and }s\in [m_\delta,M_\delta].
\end{equation}
and plugging this bound into \eqref{eq:asimov} yields
\[
\sfd(\hat \g_{t_1},\tilde \g_{t_1}) \leq C_\delta|h(t_1)-h(t_0)| \leq C_\delta \int_{t_0}^{t_1}|h'(t)|\,\d t.
\]
It is now sufficient to combine this inequality with $\sfd(\g_{t_0},\g_{t_1}) \leq \int_{t_0}^{t_1}|\dot \g_t|\d t$ and \eqref{eq:triangular} to get
\begin{equation}
\label{eq:estimate_speed_gammadot}
\sfd(\tilde \g_{t_0},\tilde \g_{t_1}) \leq \int_{t_0}^{t_1} \Big(e^{\lambda^- M_\delta}|\dot \g_t| + C_\delta|h'(t)|\Big)\,\d t.
\end{equation}
As $e^{\lambda^- M_\delta}|\dot \g_t| + C_\delta|h'(t)| \in L^1(\delta,1-\delta)$ and $\delta$ is arbitrary, the fact that $(\tilde \g_t) \in AC_{loc}((0,1))$ is thus proved. 

Turning now to the continuity of $(\tilde \g_t)$ at the endpoints, let $t_0=0$ and $t_1 \in (0,1)$.
Arguing as for \eqref{eq:triangular} but with a crucial difference in the choice of the third point in the triangular inequality, it holds
\[
\begin{split}
\sfd(\tilde \g_0,\tilde \g_{t_1}) 
& \leq
\sfd(\tilde \g_0,\S_{h(t_1)}\g_0) + \sfd(\S_{h(t_1)}\g_0,\tilde \g_{t_1}) \\
& \leq
\sfd(\S_{h(0)} \g_0,\S_{h(t_1)}\g_0) + e^{-\lambda h(t_1)}\sfd(\g_0,\g_{t_1}).
\end{split}
\]
The second term on the right-hand side vanishes as $t_1 \downarrow 0$ by (absolute) continuity of $\g$ and so does the first one, since $s \mapsto \S_s\g_0$ is continuous in $[0,\infty)$ with values in $\X$ and $h(t_1) \to h(0)$. The continuity at $t=1$ is obtained similarly and the proof is complete.
\end{proof}
\begin{rmk}{\rm
If $h(t)>0$ also in $t=0,1$, then the previous argument can be extended to the whole interval $[0,1]$ and therefore $(\tilde \g_t) \in AC([0,1],\X)$.
}\fr 
\end{rmk}

\begin{lem}
\label{lem:Vqtilda_is_AC}
With the same assumptions and notations as in Lemma \ref{lem:qtilda_is_AC}, the entropy is locally absolutely continuous in $(0,1)$ along the regularized curve $\tilde \g_t$, i.e.\
\[
t\mapsto \V(\tilde \g_t) \quad \in \quad AC_{loc}((0,1)).
\]
If in addition $(\g_t) \in AC^2([0,1],\X)$, $\V(\tilde \g_0),\V(\tilde \g_1) < \infty$ and $h$ is differentiable at $t=0$ and $t=1$ with $h'(0) > 0$ and $h'(1) < 0$, then
\[
t\mapsto \V(\tilde \g_t) \quad \in \quad C([0,1]).
\]
\end{lem}

Note that $\V(\tilde \g_0),\V(\tilde \g_1) < \infty$ is automatically satisfied if $h(t)>0$ also in $t=0,1$.

\begin{proof}
Let us first prove that $t\mapsto \V(\tilde \g_t)$ is locally absolutely continuous.
Since $|\partial \V|$ is a strong upper-gradient, the chain rule \eqref{eq:strong upper gradient} holds and it suffices to show that $|\partial\V|(\tilde \g_t)|\dot{\tilde{\g}}_t| \in L^1_{loc}(0,1)$, namely
\begin{equation}\label{eq:finiteness}
\int_\delta^{1-\delta} |\partial\V|(\tilde \g_t)|\dot{\tilde{\g}}_t|\,\d t < \infty, \qquad \forall \delta \in (0,1/2),
\end{equation}
as this would imply that $\V \circ \tilde{\g} \in AC_{loc}((0,1))$ with
\[
\left|\ddt (\V \circ \tilde \g)(t)\right| \leq |\partial \V|(\tilde{\g}_t) \cdot |\dot{\tilde{\g}}_t|, \qquad \textrm{for a.e. } t \in (0,1).
\]
To this aim, observe from \eqref{eq:estimate_speed_gammadot} that $|\dot{\tilde{\g}}_t| \in L^1_{loc}(0,1)$ with $|\dot{\tilde{\g}}_t| \leq e^{\lambda^- M_\delta}|\dot \g_t| + C_\delta|h'(t)|$ a.e.\ on $[\delta,1-\delta]$, with $M_\delta$ defined in \eqref{eq:min max}.
Moreover from \eqref{eq:estimate_slope_uniform} we also know that $|\partial\V|(\S_s\g_t) \leq C_\delta$ uniformly in $t \in [\delta,1-\delta]$ and $s \in [m_\delta,M_\delta]$, so that by choosing $s=h(t)$ we get in particular $|\partial\V|(\tilde \g_t) \leq C_\delta$ for all $t \in [\delta,1-\delta]$.
This shows that $t \mapsto |\partial\V|(\tilde \g_t)$ belongs to $L^\infty_{loc}(0,1)$, whence \eqref{eq:finiteness}.

Now assume that $(\g_t) \in AC^2([0,1],\X)$, $\V(\tilde \g_0) < \infty$, $h$ is differentiable at $t=0$ with $h'(0) > 0$ and let us prove that $t\mapsto \V(\tilde \g_t)$ is continuous at $t=0$.
(The argument is identical for $t=1$.)
On the one hand, as $(\tilde \g_t)$ is continuous at $t=0$ by Lemma \ref{lem:qtilda_is_AC} and $\V$ is lower semicontinuous, we see that $\V(\tilde \g_0) \leq \liminf_{t \downarrow 0}\V(\tilde \g_t)$.
On the other hand, choosing $t_0 = 0$ in \eqref{eq:pseudo_Riemannian_estimate_discrete}, our assumption that $h'(0)>0$ gives $h(t_1)>h(0)$ for $t_1>0$ small, hence $t^-=0$ and $t^+=t_1$.
Discarding the first two (non-negative) terms on the left-hand side, and multiplying by $(t_1-t_0)=t_1$ yield
\[
\begin{split}
\frac{1-e^{-\lambda (h(t_1)-h(0))}}{\lambda (t_1-0)} \cdot \big(\V(\tilde \g_{t_1})-\V(\tilde \g_{0})\big) & \leq \frac{t_1}{2} e^{-\lambda (h(t_1)+h(0))}\left|\frac{\sfd(\g_{t_1},\g_0)}{t_1}\right|^2 \\
& \leq \frac{t_1}{2} e^{-\lambda (h(t_1)+h(0))}\Big(\frac{1}{t_1}\int_0^{t_1} |\dot \g_t|\,\d t\Big)^2 \\
& \leq \frac{1}{2} e^{-\lambda (h(t_1)+h(0))} \int_0^{t_1} |\dot \g_t|^2\,\d t.
\end{split}.
\]
Letting $t_1 \downarrow 0$, the right-hand side vanishes owing to our assumption that $(\g_t) \in AC^2([0,1],\X)$, and clearly the exponential difference quotient in the left-hand side converges to $h'(0)$.
Rearranging gives
\[
h'(0)\limsup_{t_1 \downarrow 0} \V(\tilde \g_{t_1}) \leq h'(0) \V(\tilde \g_0),
\]
since $h'(0) > 0$ the desired upper semicontinuity follows and the proof is complete.
\end{proof}

Gathering the results proven so far, we deduce the following:

\begin{prop}
\label{prop:pseudo_Riemannian_estimate}
With the same assumptions and notations as in Lemma \ref{lem:qtilda_is_AC}, for a.e.\ $t \in (0,1)$ it holds
\begin{equation}
\label{eq:pseudo_Riemannian_estimate}
\frac{1}{2}\left|\dot{\tilde \g}_t\right|^2
+
\frac{1}{2}|h'(t)|^2 |\partial\V|^2(\tilde \g_{t})
+
h'(t)\ddt \V(\tilde \g_{t})
\leq 
\frac{1}{2}e^{-2\lambda h(t)}\left|\dot \g_t\right|^2.
\end{equation}
\end{prop}

\begin{proof}
The argument simply consists in taking the limit $t_1\to t_0$ in \eqref{eq:pseudo_Riemannian_estimate_discrete}, which should clearly lead (at least formally) to \eqref{eq:pseudo_Riemannian_estimate} by Taylor-expanding the various exponential difference quotients.
In order to make this rigorous, note that the first and third terms in the left-hand side of \eqref{eq:pseudo_Riemannian_estimate} are well defined for a.e.\ $t\in (0,1)$ by Lemma \ref{lem:qtilda_is_AC} and Lemma \ref{lem:Vqtilda_is_AC}, respectively.
The second term is also unambiguously defined because $h(t)>0$, hence the ``vertical'' $\EVI_\lambda$-gradient flow starting from $\g_t$ and defining $\tilde \g_t=\S_{h(t)}\g_t$ falls immediately within the domain $\X_1=D(|\partial \V|)$.
The right-hand side is well defined for a.e.\ $t$ since $\g\in AC([0,1],\X)$.

After this premise, let $t\in (0,1)$ be any differentiation point for $h$, $t \mapsto \g_t$, $t \mapsto \tilde \g_t$ and $t \mapsto \V(\tilde \g_t)$, choose $t_0 = t$ in \eqref{eq:pseudo_Riemannian_estimate_discrete} and let us take the right limit $t_1\downarrow t_0$ (since we are considering a differentiability point the left and right limits exist and are equal, so there is no need to address the left limit).
From the very definition \eqref{eq:def_tdelta} of $t^\pm$ it clearly holds $t^\pm \to t_0$ as $t_1 \downarrow t$, hence the convergence of the right-hand side of \eqref{eq:pseudo_Riemannian_estimate_discrete} to the right-hand side of \eqref{eq:pseudo_Riemannian_estimate} is clear and so is the convergence of the two difference quotients of $h$.
By Lemma~\ref{lem:qtilda_is_AC} the first term in the left-hand side also passes to the limit, as does the third one according to Lemma~\ref{lem:Vqtilda_is_AC}.
The only term left to handle is the Fisher information $|\partial \V|^2(\tilde \g^+)$. From the continuity of $t \mapsto \tilde \g_t$ (cf.\ Lemma~\ref{lem:qtilda_is_AC}) we see that $\tilde \g_{t^+} \to \tilde \g_t$ in $(\X,\sfd)$, and the lower semicontinuity of the slope \eqref{eq:local=global} results in
\[
|\partial \V|(\tilde \g_{t})\leq \liminf_{t_1 \downarrow 0} |\partial \V|(\tilde \g_{t^+}).
\]
Thus rigorously taking the liminf $t_1\downarrow t_0$ in \eqref{eq:pseudo_Riemannian_estimate_discrete} entails \eqref{eq:pseudo_Riemannian_estimate} and achieves the proof.
\end{proof}
The interesting consequence for our purpose is then:
\begin{theo}
\label{theo:upper_bound_recovery_metric}
With the same assumptions and notations as in Setting \ref{setting}, fix $\eps>0$, and set $h_\eps(t) := \eps \min\{t,1-t\}$. Let $(\g_t)\in AC^2([0,1],\X)$ be such that $\V(\g_0),\V(\g_1)<\infty$ and define
\[
\g^\eps_t := \S_{h_\eps(t)}\g_t, \qquad t \in [0,1].
\]
Then $(\g^\eps_t) \in AC^2([0,1],\X)$, $t \mapsto \V(\g^\eps_t)$ belongs to $AC([0,1])$ and it holds
\begin{equation}
\label{eq:upper_bound_recovery_metric}
\begin{split}
\frac{1}{2}\int_0^1|\dot \g^\eps_t|^2\,\dt + \frac{\eps^2}{2}\int_0^1 |\partial \V|^2(\g^ \eps_t)\,\dt 
& \leq \frac{1}{2}e^{\lambda^-\eps}\int_0^1|\dot \g_t|^2\,\dt -2\eps\V(\g^\eps_{1/2}) \\
& \qquad + \eps \Big(\V(\g_0) + \V(\g_1)\Big).
\end{split}
\end{equation}
\end{theo}
\noindent
Note here that $h_\eps(0)=h_\eps(1)=0$, so that the endpoints $\g^\eps_0=\g_0$ and $\g^\eps_1=\g_1$ remain unchanged.
\begin{proof}
The strategy of proof simply consists in integrating \eqref{eq:pseudo_Riemannian_estimate} between $0$ and $1$ while integrating by parts of the term $h_\eps'(t)\ddt \V(\g^ \eps_t)$, separately on $[0,1/2]$ and $[1/2,1]$.
Note carefully that our specific choice gives $h_\eps' = \eps$ and $h_\eps' = -\eps$ on these two time intervals, respectively.
Taking into account $e^{-2\lambda h_\eps(t)} \leq e^{\lambda^-\eps}$, where $\lambda^- := \max\{-\lambda,0\}$, this procedure yields
\[
\begin{split}
\frac{1}{2}\int_0^1|\dot \g^ \eps_t|^2\,\dt + \frac{\eps^2}{2}\int_0^1 |\partial\V|^2(\g^ \eps_t)\,\dt & \leq \frac{1}{2}e^{\lambda^-\eps}\int_0^1|\dot \g_t|^2 \,\dt -2\eps\V(\g^\eps_{1/2}) \\
& \quad + \eps \Big(\V(\g^\eps_0)+\V(\g^\eps_1)\Big).
\end{split}
\]
The term $2\eps \V(\g^\eps_{1/2})$ simply arises from the two boundary terms at $t=1/2$ in the two integrations by parts.
(Alternatively, it can be seen as the result of $-\int_0^1 \V(\g^\eps_t)h''(t)$ arising from the integration by parts in the whole interval $[0,1]$, with the singularity $h''(t)=-2\eps\delta_{1/2}(t)$).
However, this argument is not fully rigorous because all the terms on the left-hand side of \eqref{eq:pseudo_Riemannian_estimate} are only \emph{locally} integrable, hence we may not be allowed to integrate them all the way to $t=0$ and $t=1$.

In order to circumvent this slight issue, choose $\delta \in (0,1/2)$ and carry out the same argument on $[\delta,1/2]$ and $[1/2,1-\delta]$ rather than on $[0,1/2]$ and $[1/2,1]$: Integration by parts is now justified by Lemma \ref{lem:Vqtilda_is_AC} and this provides us with
\begin{equation}
\label{eq:upper_bound_recovery_metric_delta}
\begin{split}
\frac{1}{2}\int_{\delta}^{1-\delta}|\dot \g^\eps_t|^2\,\dt + \frac{\eps^2}{2}\int_{\delta}^{1-\delta} |\partial \V|^2(\g^ \eps_t)\,\dt & \leq \frac{1}{2}e^{\lambda^-\eps}\int_{\delta}^{1-\delta}|\dot \g_t|^2\,\dt \\
& \quad + \eps \Big(\V(\g^\eps_{\delta})-2\V(\g^\eps_{1/2})+\V(\g^\eps_{1-\delta})\Big).
\end{split}
\end{equation}
It is then sufficient to pass to the limit as $\delta \downarrow 0$.
By monotonicity the left-hand side above converges to the left-hand side in \eqref{eq:upper_bound_recovery_metric} and for the same reason so does the first term on the right-hand side, while by the current choice of $h$ and by Lemma \ref{lem:Vqtilda_is_AC} $t \mapsto \V(\g_t^\eps)$ is continuous on the whole interval $[0,1]$, so that
\[
\lim_{\delta \downarrow 0} \eps \Big(\V(\g^\eps_{\delta})+\V(\g^\eps_{1-\delta})\Big) 
= \eps \Big(\V(\g^\eps_0)+\V(\g^\eps_1)\Big)
=\eps \Big(\V(\g_0)+\V(\g_1)\Big)
\]
and \eqref{eq:upper_bound_recovery_metric} follows.

Finally, since the right-hand side of \eqref{eq:upper_bound_recovery_metric} is finite we see that $|\dot\g^\eps| \in L^2(0,1)$ and $|\partial \V|(\g^ \eps) \in L^2(0,1)$. As a consequence $|\dot\g^\eps|\cdot|\partial\V(\g^\eps)|\in L^1(0,1)$ in the strong upper-chain rule \eqref{eq:strong upper gradient}, and $\V \circ \g^\eps \in AC([0,1])$ as desired.
\end{proof}
\section{Small-temperature limit and convexity} 
\label{sec:small_noise_convexity}
\subsection{\texorpdfstring{$\Gamma$}{Gamma}-convergence of the Schr\"odinger problem}

Relying on the results of the previous section, we can now turn to Theorem \ref{theo:Gamma_lim_formal} and make it rigorous in the metric setting. To this end, let us first introduce two action functionals: the kinetic energy $\mathcal{A}$ and the (halved) Fisher information $\mathcal{I}$ along a curve, respectively defined as
\[
\mathcal{A}(\g) := \frac{1}{2}\int_0^1|\dot \g_t|^2\,\dt \qquad\textrm{and}\qquad \mathcal{I}(\g) := \frac{1}{2}\int_0^1 |\partial \V|^2(\g_t)\,\dt
\]
for all $(\g_t) \in C([0,1],\X)$, where it is understood that $\mathcal{A}(\g) = +\infty$ whenever $\g$ is not absolutely continuous. Given two points $x,y \in \X$ and a temperature/slowing-down parameter $\eps>0$, the (metric) Schr\"odinger problem reads as
\begin{equation}
\label{eq:schrodinger_pb}
\inf_{(\g_t) \,:\, x \leadsto y}\Big\{\mathcal{A}(\g) + \eps^2\mathcal{I}(\g)\Big\},
\tag{$\mathrm{Sch}^\eps$}
\end{equation}
where $(\g_t) \,:\, x \leadsto y$ is a short-hand notation meaning that the infimum runs over all $(\g_t) \in C([0,1],\X)$ such that $\g_0 = x$ and $\g_1 = y$. For sake of brevity we also introduce
\[
\mathcal{A}_\eps(\g) := \mathcal{A}(\g) + \eps^2\mathcal{I}(\g).
\]
From \eqref{eq:schrodinger_pb} it is thus clear that the Fisher information $\mathcal{I}$ acts as a perturbation of $\mathcal{A}$ and this has a regularizing effect, since minimizers of \eqref{eq:schrodinger_pb} live within the regular domain $\X_1 = D(|\partial\V|)$.

\begin{rmk}{\rm
The smoothing effect is well understood for the classic Schr\"odinger problem in a regular setting, namely when $\V$ is the Boltzmann-Shannon relative entropy and $\X$ is the Wasserstein space over a smooth Riemannian manifold. In this case, under mild assumptions on the end-points, minimizers of \eqref{eq:schrodinger_pb} are curves of absolutely continuous measures whose densities are bounded, smooth, Lipschitz, with exponentially fast decaying tails.

In the current metric framework the properties above are meaningless, but still minimizers of \eqref{eq:schrodinger_pb} are ``regular'' from a metric point of view, since as just said they live within $D(|\partial\V|)$. Moreover, in Proposition \ref{prop:schro_solvable_iff_finite_entropy} we are going to see that $\V$ is absolutely continuous along optimal curves.
}\fr 
\end{rmk}

Let us first deal with the solvability of \eqref{eq:schrodinger_pb}.

\begin{prop}
\label{prop:schro_solvable_iff_finite_entropy}
With the same assumptions and notations as in Setting \ref{setting} and under Assumption \ref{hyp1}, for any fixed $x,y \in \X$ and $\eps>0$ the Schr\"odinger problem \eqref{eq:schrodinger_pb} is solvable if and only if $\V(x),\V(y)<\infty$ and there exists $(\g_t) \in AC([0,1],\X)$ such that $\g_0=x$ and $\g_1=y$.
\end{prop}

As the condition characterizing the solvability of the Schr\"odinger problem does not depend on $\eps$, it is clear that if \eqref{eq:schrodinger_pb} is solvable for some $\eps>0$, then it is actually solvable for all $\eps>0$.

\begin{proof}
Assume that the endpoints have finite entropy and that there exists an absolutely continuous curve $\g$ connecting $x$ to $y$. Up to reparametrization, we can assume that $(\g_t) \in AC^2([0,1],\X)$.
Theorem~\ref{theo:upper_bound_recovery_metric} thus guarantees that $\mathcal{A}_\eps$ is finite along the regularization $(\g^ \eps_t)_{t\in[0,1]}$ of this curve and therefore the variational problem \eqref{eq:schrodinger_pb} is proper.
Let then $(\g^n_t)$ be any minimizing sequence and observe that the kinetic action $\mathcal{A}$ is bounded uniformly in $n$, say $\mathcal{A}(\g^n) \leq C$ for all $n$. We now observe that for any pair $0 \leq t_0 < t_1 \leq 1$ it holds
\begin{equation}
\label{eccurve1}
\sfd(\g^n_{t_0},\g^n_{t_1}) \leq \int_{t_0}^{t_1}|\dot \g^n_t| \,\dt \leq |t_0-t_1|^{1/2}\Big(\int_{t_0}^{t_1}|\dot \g^ n_t|^2 \,\dt\Big)^{1/2} \leq C|t_0-t_1|^{1/2}.
\end{equation} 
Since the endpoints are fixed, this implies that the set of points $\g^n_t$ is bounded in $(\X,\sfd)$ uniformly in $n,t$, thus it is $\sigma$-relatively sequentially compact by Assumption \ref{hyp1}. By the refined Arzel\`a-Ascoli lemma \cite[Proposition 3.3.1]{AmbrosioGigliSavare08}, there exists a limiting $\sfd$-continuous (actually $1/2$-H\"older continuous) curve $\g$ such that 
\[
\g^n_t \stackrel{\sigma}{\to} \g_t, \qquad \forall t \in [0,1].
\]
We now observe that the kinetic action is lower semicontinuous for this pointwise-in-time convergence w.r.t.\ $\sigma$, cf.\ \cite[Section 2.2]{AmbrosioGigliSavare11} (indeed, $\sfd$ is lower semicontinuous w.r.t.\ $\sigma$, hence the $2$-energies of the finite partitions of $\g$ are lower semicontinuous w.r.t.\ $\sigma$ too, whence the lower semicontinuity of the $2$-energy of $\g$ itself).
Moreover $|\partial \V|^2$ is also lower semicontinuous w.r.t.\ $\sigma$ by hypothesis, and this fact together with Fatou's lemma gives
\[
\int_0^1 |\partial \V|^2(\g_t)\,\dt \leq \int_0^1 \liminf_{n \to \infty} |\partial \V|^2(\g^ n_t)\,\dt \leq \liminf_{n \to \infty} \int_0^1 |\partial \V|^2(\g^ n_t)\,\dt.
\]
Therefore $\g$ is a minimizer of \eqref{eq:schrodinger_pb}.

Conversely, assume that there exists a minimizer, denoted by $\g$ (the following argument actually works for any curve along which $\mathcal{A}_\eps$ is finite and without Assumption \ref{hyp1}).
Then in particular $t \mapsto |\dot \g_t|$ and $t \mapsto |\partial \V|(\g_t)$ belong to $L^2(0,1)$ and by \eqref{eq:strong upper gradient} we see that $t \mapsto \V(\g_t)$ is globally absolutely continuous with
\[
\left|\ddt(\V \circ \g)(t)\right|\leq |\partial \V|(\g_t) \cdot |\dot \g_t| \in L^1(0,1).
\]
The fact that $(|\dot \g_t|) \in L^2(0,1) \subset L^1(0,1)$ trivially implies $(\g_t) \in AC([0,1],\X)$, whereas the fact that $t \mapsto |\partial \V|(\g_t)$ belongs to $L^2(0,1)$ also implies that $|\partial \V|(\g_t)$ is finite for a.e.\ $t \in [0,1]$ and a fortiori so is $\V(\g_t)$, since $D(|\partial \V|)\subset D(\V)$.
Hence let $t^* \in (0,1)$ be any point satifying $\V(\g_{t^*})<\infty$ and note that together with \eqref{eq:strong upper gradient} this gives the following global upper bound valid for all $t < t^*$
\[
\V(\g_t) \leq \V(\g_{t^*}) + \int_t^{t^*} \left|\ddt(\V \circ \g)(t)\right|\dt \leq \V(\g_{t^*}) + \int_0^1 \left|\ddt(\V \circ \g)(t)\right|\dt =: \overline{\V} < \infty.
\]
As a consequence, and taking also into account the facts that $t \mapsto \g_t$ is $\sfd$-continuous and $\V$ is lower semicontinuous, we get
\[
\V(\g_0) = \V\big(\lim_{t \to 0} \g_t\big) \leq \liminf_{t \to 0} \V(\g_t) \leq \overline{\V}
\]
and the proof is thus complete, as the same argument applies \emph{mutatis mutandis} for $t=1$ too.
\end{proof}

We now fix $x,y \in \X$ and let $C([0,1],\X) \ni \g \mapsto \iota_{01}(\g)$ denote the convex indicator of the endpoint constraints, i.e.\
\[
\iota_{01}(\g) =
\begin{cases}
  0 & \mbox{if }\g_0 = x \mbox{ and }\g_1 = y,\\
  +\infty &\mbox{otherwise.}
\end{cases}
\]
With this said, we can finally state our $\Gamma$-convergence result, where the finite-entropy assumption on the endpoints is motivated by the previous proposition. 
\begin{theo}\label{t:gconv1}
With the same assumptions and notations as in Setting \ref{setting}, if $x,y \in \X$ are such that $\V(x),\V(y) < \infty$, then
\[
\Gamma-\lim_{\eps \to 0}\Big\{ \mathcal A_\eps +\iota_{01}\Big\} = \mathcal A + \iota_{01}
\]
for the uniform convergence on $C([0,1],\X)$.
If Assumption \ref{hyp1} holds, then the $\Gamma$-convergence also takes place w.r.t.\ the pointwise-in-time $\sigma$-topology. 
\end{theo}

\begin{proof}
The $\Gamma-\liminf$ inequality is rather clear, since the kinetic energy $\mathcal{A}$ is lower semicontinuous both w.r.t.\ uniform-in-time $\sfd$-convergence and pointwise-in-time $\sigma$-convergence:
for the former topology the fact is well known, for the latter it has been discussed in the proof of Proposition \ref{prop:schro_solvable_iff_finite_entropy}.
An analogous claim is also true for the convex indicator $\iota_{01}$.
As a consequence, we have that for any $\g^ \eps$ converging to $\g$ uniformly in time in the metric topology or pointwise in time in the topology $\sigma$ (if applicable) it holds
\[
\begin{split}
\mathcal A(\g)+\iota_{01}(\g) 
& \leq \liminf_{\eps \downarrow 0} \mathcal A(\g^ \eps) + \liminf_{\eps \downarrow 0}\iota_{01}(\g^ \eps) \leq \liminf_{\eps \downarrow 0} \Big\{\mathcal A(\g^ \eps)+\iota_{01}(\g^ \eps)\Big\} \\
& \leq \liminf_{\eps \downarrow 0} \Big\{\mathcal A(\g^ \eps)+\eps^2\mathcal I(\g^ \eps)+\iota_{01}(\g^ \eps)\Big\} = \liminf_{\eps \downarrow 0} \Big\{ \mathcal A_\eps(\g^ \eps)+\iota_{01}(\g^ \eps)\Big\},
\end{split}
\]
whence the desired $\Gamma-\liminf$ inequality.

For the $\Gamma-\limsup$, take any $(\g_t) \in AC^2([0,1],\X)$ connecting $x$ to $y$ (if it does not exist, then there is nothing to prove).
Then Theorem~\ref{theo:upper_bound_recovery_metric} precisely provides a recovery sequence $\g^\eps_t := \S_{h_{\eps}(t)}\g_t$ with $h_\eps$ defined as therein, both for the uniform-in-time $\sfd$-convergence and the pointwise-in-time $\sigma$-convergence (the latter is an easy consequence of the former by Remark \ref{rmk19}).
To prove this claim, note that for any $n \in \mathbb{N}$ there exist $t_1,...,t_k \in [0,1]$ such that, for any $t \in [0,1]$, $\sfd(\g_t,\g_{t_i}) < 1/n$ for at least one $t_i$;
in addition, since $\g_t^\eps \to \g_t$ for all $t \in [0,1]$ there exists $\eps_n$ small enough such that $\sfd(\g_{t_i},\g_{t_i}^\eps) < 1/n$ for all $\eps < \eps_n$ and $i=1,...,k$.
As a consequence, taking \eqref{eq:contraction} into account,
\[
\begin{split}
\sfd(\g_t,\g_t^\eps) & \leq \sfd(\g_t,\g_{t_i}) + \sfd(\g_{t_i},\S_{h_\eps(t)}\g_{t_i}) + \sfd(\S_{h_\eps(t)}\g_{t_i},\g_t^\eps) \\
& \leq \sfd(\g_t,\g_{t_i}) + \sfd(\g_{t_i},\S_{h_\eps(t)}\g_{t_i}) + e^{-\lambda h_\eps(t)}\sfd(\g_t,\g_{t_i}) \\
& \leq \frac{1}{n}(2+e^{\lambda^-\eps/2})
\end{split}
\] 
for all $t \in [0,1]$ and $\eps < \eps_n$ and by the arbitrariness of $n$ we conclude that $\g^\eps \to \g$ uniformly.
Furthermore, the $\limsup$ inequality can be proved as follows:
\[
\begin{split}
\limsup_{\eps \downarrow 0}\Big\{ \mathcal A_\eps(\g^ \eps) & + \iota_{01}(\g^ \eps)\Big\} = \limsup_{\eps \downarrow 0} \Big\{\mathcal A(\g^ \eps)+\eps^2\mathcal I(\g^ \eps)+0\Big\} \\
& \overset{\eqref{eq:upper_bound_recovery_metric}}{\leq} \limsup_{\eps \downarrow 0} \Big\{e^{\lambda^-\eps} \mathcal A(\g) - 2\eps\V(\g^\eps_{1/2}) + \eps\big(\V(x)+\V(y)\big) \Big\} \\
& \leq \limsup_{\eps \downarrow 0} \Big\{e^{\lambda^-\eps} \mathcal A(\g) + \eps\big(\V(x)+\V(y)\big) \Big\} - 2\liminf_{\eps \downarrow 0}\eps\V(\g^\eps_{1/2}) \\
& \leq \mathcal A(\g) = \mathcal A(\g)+\iota_{01}(\g),
\end{split}
\]
where the third inequality comes from the fact that, for any $\eps \downarrow 0$, $(\g_{1/2}^\eps)$ is contained in a bounded set and by assumption $\V$ is bounded from below on bounded sets, whence $\V(\g^\eps_{1/2}) \geq c$ for some $c \in \mathbb{R}$.
The proof is thus complete.
\end{proof}

As an easy consequence of this result we obtain the following:

\begin{cor}
\label{c:gc}
With the same assumptions and notations as in Setting \ref{setting} and under the further requirements that Assumption \ref{hyp1} holds and the Schr\"odinger problem \eqref{eq:schrodinger_pb} relative to $x,y \in \X$ is solvable, let $\eps_k \downarrow 0$ and $\omega^k$ be a minimizer of the corresponding Schr\"odinger problem \eqref{eq:schrodinger_pb} with $\eps=\eps_k$. 

Then
\[
\lim_{k \to \infty}\Big\{\mathcal{A}(\omega^k) + \eps_k^2\mathcal{I}(\omega^k)\Big\} = \inf_{(\g_t) \,:\, x \leadsto y} \mathcal{A}(\g).
\]
Moreover, there exists $\omega^0 \in C([0,1],\X)$ such that, up to a subsequence, $\omega^k \to \omega^0$ in the pointwise-in-time $\sigma$-topology and
\[
\mathcal{A}(\omega^0) = \inf_{(\g_t) \,:\, x \leadsto y} \mathcal{A}(\g).
\]
\end{cor}

\begin{proof}
Recall that, under a mild equicoercivity condition, $\Gamma$-convergence precisely guarantees that the limit of the optimal values of the approximating problems is the optimal value of the limit problem and limits of minimizers are minimizers, cf.\ \cite[Theorem 1.21]{Braides02}.
In view of Theorem~\ref{t:gconv1} and \cite[Theorem 1.21]{Braides02}, for the mild equi-coercivity condition to hold it suffices to prove that the set of minimizers $\{\omega^k\}$ is relatively compact in the pointwise-in-time $\sigma$-topology.
To this aim, the kinetic energies of the curves $\omega^k$ are uniformly bounded since
\[
\mathcal A(\omega^k) \leq \mathcal A(\omega^k) + \eps_k^2 \mathcal I(\omega^k) \leq \mathcal
 A(\omega^{\overline\eps}) + \eps_k^2 \mathcal I(\omega^{\overline\eps}) \leq \mathcal A(\omega^{\overline\eps}) + \overline{\eps}^2 \mathcal I(\omega^{\overline\eps})<+\infty,
\]
where $\omega^{\overline\eps}$ is the minimizer for the problem with $\eps=\overline\eps := \sup_k{\eps_k}$. Arguing as in the proof of Proposition \ref{prop:schro_solvable_iff_finite_entropy}, we deduce that there exists a continuous curve $(\omega^0_t)_{t\in [0,1]}$ connecting $x$ and $y$ such that, up to extracting a suitable subsequence, $\omega^k_t\to \omega^0_t$ w.r.t.\ $\sigma$ as $k \to \infty$ for all $t \in [0,1]$.
\end{proof}

\begin{rmk}{\rm
Note that in Corollary \ref{c:gc} the curve $\omega^0$ is length-minimizing but not necessarily distance-minimizing, namely it needs not be a geodesic between $x$ and $y$, since we only know that
\[
\inf_{(\g_t) \,:\, x \leadsto y} \mathcal{A}(\g) \geq \frac{1}{2}\sfd^2(x,y)
\]
and the inequality might be strict, e.g.\ if $\X$ is a non-convex subset of $\mathbb{R}^d$. 
However, if $(\X,\sfd)$ is a length metric space, i.e.\ for all $x,y \in \X$ and $\eps > 0$ there exists $(\g_t) \in AC([0,1],\X)$ such that $\g_0 = x$, $\g_1 = y$ and $\ell(\g) \leq \sfd(x,y)+\eps$, then the inequality above turns out to be an identity and, as a consequence, $\omega^0$ is a geodesic.
This means that for any two points having finite energy there always exists a geodesic connecting them.
}\fr 
\end{rmk}

When the endpoints have infinite entropy, the following variant of Theorem \ref{t:gconv1} may be useful:

\begin{theo} \label{t:gconv2}
With the same assumptions and notations as in Setting \ref{setting}, let $x,y \in \X$ with possibly $\V(x),\V(y)=+\infty$ and for any fixed $(\eps_n)_{n \in \mathbb{N}}$, $\eps_n \downarrow 0$, let $(\eta_n)_{n \in \mathbb{N}}$ be converging to 0 slowly enough so that
$$
\eps_n \left(\V(\g_0^n)+\V(\g_1^n\right))\to 0
\qqtext{with}
\g_0^n:=\S_{\eta_n}x,\quad \g_1^n:=\S_{\eta_n}y.
$$
Then
\[
\Gamma-\lim\limits_{n\to\infty}\Big\{ \mathcal A_{\eps_n} +\iota^n_{01}\Big\} = \mathcal A +\iota_{01},
\] 
for the uniform convergence on $C([0,1],\X)$.
If Assumption \ref{hyp1} holds, then the $\Gamma$-convergence also takes place w.r.t. the pointwise-in-time $\sigma$-topology.
Here $\iota_{01}^n$ and $\iota_{01}$ are the convex indicators of the endpoint constraints for $\g_0^n,\g_1^n$ and $x,y$, respectively.
\end{theo}

\begin{proof} 
The proof of the $\Gamma-\liminf$ is almost identical to that in Theorem~\ref{t:gconv1}, with the only extra observation that 
\[
\iota_{01}(\g) \leq \liminf_{n \to \infty} \iota_{01}^n(\g^ n).
\] 
For the $\Gamma-\limsup$, observe that if there does not exist $(\g_t) \in AC^2([0,1],\X)$ joining $x$ and $y$, then there is nothing to prove.
Hence let us suppose that at least one curve $(\g_t) \in AC^2([0,1],\X)$ connecting $x$ and $y$ exists, fix it and note that Theorem~\ref{theo:upper_bound_recovery_metric} applied to the curve $\S_{\eta_n}\g_t$ still provides a recovery sequence $\g^{\eps_n}_t := \S_{\eta_n+h_{\eps_n}(t)}\g_t$ with the same choice $h_{\eps_n}(t)=\eps_n\min\{t,1-t\}$ as before.
Indeed, on the one hand 
\[
\begin{split}
\limsup_{n \to\infty}\Big\{ & \mathcal A_{\eps_n}(\g^ {\eps_n}) + \iota_{01}^n(\g^ {\eps_n})\Big\} 
= \limsup_{n \to\infty} \Big\{\mathcal A(\g^ {\eps_n}) + \eps^2_n\mathcal I(\g^ {\eps_n}) + 0 \Big\} 
\\
& \stackrel{\eqref{eq:upper_bound_recovery_metric}}{\leq} \limsup_{n \to\infty} \Big\{ e^{\lambda^-\eps_n}\mathcal A(\S_{\eta_n} \g) - 2\eps_n \V(\g_{1/2}^{\eps_n}) + \eps_n\big(\V(\g_0^n)+\V(\g_1^n)\big) \Big\}
\\
& \leq \limsup_{n \to\infty} \Big\{ e^{\lambda^-\eps_n}\mathcal A(\S_{\eta_n} \g) + \eps_n\big(\V(\g_0^n)+\V(\g_1^n)\big) \Big\} - 2\liminf_{\eps \downarrow 0}\eps\V(\g^\eps_{1/2})
\\
& \leq \limsup_{n \to\infty} \mathcal A(\S_{\eta_n} \g) \leq \limsup_{n \to\infty} e^{-2\lambda \eta_n}\mathcal A(\g) = \mathcal A(\g),
\end{split}
\]
where the third inequality follows by the same argument adopted in the proof of the previous theorem and the last one is due to \eqref{eq:pseudo_Riemannian_estimate} with $h(t) \equiv \eta_n$. 
On the other hand, $\g^{\eps_n}_t \to \g_t$ uniformly in $t \in [0,1]$	in the $\sfd$-topology and, if Assumption \ref{hyp1} holds, for all $t \in [0,1]$ w.r.t.\ $\sigma$: the argument described in the previous proof applies also here verbatim.
\end{proof}

As conclusion, in the next proposition we show that any $\EVI$-gradient flow is a solution of the Schr\"odinger problem with suitable endpoints.
Intuitively this is clear, because up to a rescaling factor $\eps^2$ both the trajectories of the gradient flow of $\V$ and the solutions to \eqref{eq:schrodinger_pb} must formally satisfy the same Newton equation, namely $\ddot \g_t=-\nabla \Phi(\g_t)$ where the potential $\Phi$ is given by (minus) the Fisher information $-|\partial \V|^2$, cf. \cite[Remark 6]{GLR18}.
This is also in complete analogy with the standard Schr\"odinger problem, which includes the heat flow as a particular entropic interpolation.

\begin{prop} \label{p:gfs}
With the same assumptions and notations as in Setting \ref{setting}, fix $\eps > 0$.
Then for all $x,y \in \X$ the following lower bound on the optimal value of \eqref{eq:schrodinger_pb} holds
\begin{equation}\label{eq:lower bound}
\inf_{(\g_t) \,:\, x \leadsto y}\mathcal{A}_\eps(\g) \geq \eps\big|\V(x) - \V(y)\big|.
\end{equation}
If either $y = \S_\eps x$ or $x = \S_\eps y$, then equality is achieved. 
In the former case the curve $[0,1] \ni t \mapsto \hat{\g}_t := \S_{\eps t} x$ is a minimizer in the Schr\"odinger problem and the optimal value is
\[
\inf_{(\g_t) \,:\, x \leadsto y}\mathcal{A}_\eps(\g) = \eps\big(\V(x) - \V(\S_\eps x)\big).
\]
An analogous statement holds when $x = \S_\eps y$.
\end{prop}

\begin{proof}
By \eqref{eq:strong upper gradient} and Young's inequality it follows that for any $(\tilde{\g}_t) \in AC^2([0,\eps],\X)$ joining $x$ and $y$ (if it exists; if not, \eqref{eq:lower bound} is trivial) it holds
\[
\big|\V(\tilde{\g}_0) - \V(\tilde{\g}_\eps)\big| \leq \frac{1}{2}\int_0^\eps |\dot{\tilde{\g}}_t|^2\,\dt + \frac{1}{2}\int_0^\eps |\partial \V|^2(\tilde{\g}_t)\,\dt.
\]
By setting $\g_t := \tilde{\g}_{\eps t}$, $t \in [0,1]$, and by the arbitrariness of $\tilde{\g}$ we thus see that for all $(\g_t) \in AC^2([0,1],\X)$ joining $x$ and $y$ we have
\[
\eps \big|\V(\g_0) - \V(\g_1)\big| \leq \frac12\int_0^1 |\dot{\g}_t|^2\,\dt + \frac{\eps^2}{2}\int_0^1 |\partial \V|^2(\g_t)\,\dt,
\]
so that
\[
\eps\big|\V(\g_0) - \V(\g_1)\big| \leq \inf_{(\g_t) \,:\, x \leadsto y}\mathcal{A}_\eps(\g).
\]
Now assume that $y = \S_\eps x$: integrating \eqref{eq:speed=slope} for the $\EVI_\lambda$-gradient flow $\hat{\g}$ (paying attention to the rescaling factor $\eps$) between 0 and 1 we get
\[
\mathcal{A}_\eps(\hat \g) = \frac12\int_0^1 |\dot{\hat{\g}}_t|^2\,\dt + \frac{\eps^2}{2}\int_0^1 |\partial \V|^2(\hat{\g}_t)\,\dt = \eps\big(\V(x) - \V(y)\big) = \eps\big|\V(x) - \V(y)\big|,
\]
where the last equality comes from the fact that $t \mapsto \V(\S_t x)$ is non-increasing, as a consequence of \eqref{eq:speed=slope}.
Combining this identity with \eqref{eq:lower bound} yields the conclusion.
\end{proof}

\subsection{Displacement convexity}

In analogy with Section \ref{sec:convexity-heuristic}, in this short section we establish the geodesic $\lambda$-convexity of $\V$.
As already explained in the Introduction, here \textbf{we do not claim any novelty of the result} (cf.\ \cite[Theorem 3.2]{DaneriSavare08}). Our proof is however independent and new, being a further evidence of the wide range of applications of the Schr\"odinger problem. 
Let us stress once more that all the properties of $\EVI_\lambda$-gradient flows stated in Section \ref{sec:preliminaries} and used so far do not rely on geodesic $\lambda$-convexity, whence the genuine independence of our approach.

\begin{theo}
\label{prop:V_geodesically_convex_lambda}
With the same assumptions and notations as in Setting \ref{setting}, the potential $\V$ is $\lambda$-convex along \emph{any} geodesic.
\end{theo}

\begin{proof}
Let $(\g_t)$ be any constant-speed geodesic. We want to prove that
$$
\V(\g_\theta)\leq (1-\theta)\V(\g_0) + \theta \V(\g_1) -\frac{\lambda}{2} \theta(1-\theta)\sfd^2(\g_0,\g_1), \qquad \forall \theta \in [0,1].
$$
We will establish this inequality by carefully estimating at order one as $\eps \downarrow 0$ the defect of optimality, in the geodesic problem from $\g_0$ to $\g_1$, of a suitably regularized version $(\g^ {\eps}_t)$ of the geodesic.

If $\V(\g_0)=+\infty$ or $\V(\g_1)=+\infty$ there is nothing to prove, so we can assume without loss of generality that both endpoints have finite entropy. If $\theta = 0$ or $\theta = 1$ the inequality is trivial as well. 
Fix then an arbitrary parameter $\theta\in(0,1)$ and let
\[
H_\theta(t):=
 \left\{
 \begin{array}{ll}
  \displaystyle{\frac{1}{\theta}} t & \mbox{if }t\in [0,\theta],\\
  - \displaystyle{\frac{1}{1-\theta}(t-1)} & \mbox{if }t\in [\theta,1].
 \end{array}
 \right.
\]
be the hat function centered at $t=\theta$ with height $1$ and vanishing at $t=0,1$.
Setting $h(t):=\eps H_\theta(t)$ for small $\eps>0$, let $(\g^\eps_t)$ be the curve constructed as in Lemma \ref{lem:technical formula}, i.e.
\[
\g^\eps_t := \S_{h(t)}\g_t, \qquad\mbox{for all }t\in[0,1].
\]
Arguing as in the proof of Theorem~\ref{theo:upper_bound_recovery_metric}, it is easily verified that with the current choice of $h$ it is still true that $t \mapsto |\dot \g_t^\eps|$ and $t \mapsto |\partial\V|(\g_t^\eps)$ belong to $AC^2([0,1],\X)$ and $t \mapsto \V(\g_t^\eps)$ to $AC([0,1])$, so that we can integrate \eqref{eq:pseudo_Riemannian_estimate} in time on the whole interval $[0,1]$.
Discarding the non-negative term $|h'(t)|^2|\partial \V|^2(\g^\eps_t)$ and using the optimality of the geodesic $\g$ (namely its optimality between $\g_0$ and $\g_1$) give
\[
\begin{split}
0 & \leq \frac 12 \int_0^1 |\dot \g_t^\eps|^2 \,\dt - \frac 12 \int_0^1 |\dot \g_t|^2 \,\dt 
\\
& \overset{\eqref{eq:pseudo_Riemannian_estimate}}{\leq} -\int_0^1 h'(t)\ddt \V(\g^ \eps_{t})\,\dt + \frac 12 \int_0^1 \left(e^{-2\lambda h(t)}-1\right)|\dot \g_t|^2\,\dt 
\\
& = -\eps\int_0^1 H_\theta'(t)\ddt \V(\g^ \eps_{t}) \,\dt + \frac{\sfd^2(\g_0,\g_1)}{2} \int_0^1  \left(e^{-2\eps\lambda H_\theta(t)}-1\right)\,\dt,
\end{split}
\]
where the last equality follows from the constant speed property of the geodesic $\g$, namely $|\dot \g_t| = \sfd(\g_0,\g_1)$. 
Dividing by $\eps>0$ and leveraging the explicit piecewise constant values of $H_{\theta}'(t)$ on each interval $(0,\theta)$ and $(\theta,1)$ gives
\[
\begin{split}
0 & \leq -\int_0^1 H_\theta'(t)\ddt \V(\g^ \eps_{t})\,\dt + \frac{\sfd^2(\g_0,\g_1)}{2}\underbrace{\int_0^1  \frac{e^{-2\eps\lambda H_\theta(t)}-1}{\eps}\,\dt}_{:=I_\eps}
\\
& = -\int_0^\theta\frac{1}{\theta}\ddt\V(\g^ \eps_{t})\,\dt + \int_\theta^1 \frac{1}{1-\theta} \ddt\V(\g^ \eps_{t})\,\dt + \frac{\sfd^2(\g_0,\g_1)}{2} I_\eps 
\\
& = \frac{1}{\theta}\Big( \V(\g_0) - \V(\g^ \eps_\theta)\Big) +\frac{1}{1-\theta}\Big( \V(\g_1) - \V(\g^ \eps_\theta)\Big) + \frac{\sfd^2(\g_0,\g_1)}{2} I_\epsilon.
\end{split}
\]
Now let us multiply by $\theta(1-\theta)>0$ and rearrange the terms in order to get
\[
\V(\g^\eps_\theta) \leq (1-\theta)\V(\g_0) + \theta\V(\g_1) +\theta(1-\theta)\frac{\sfd^2(\g_0,\g_1)}{2}I_\eps.
\]
It is easy to check that $\int_0^1 H_{\theta}(t)\d t=\frac 12$ for all $\theta$, so that
\[
\lim_{\eps \downarrow 0} I_\eps = -2\lambda\int_0^1 H_\theta(t)\,\dt = -\lambda.
\]
On the other hand, by definition of $\g^ \eps$ and since $h(\theta)=\eps\to 0$ it is clear that $\g^\eps_\theta = \S_{h(\theta)}\g_\theta = \S_\eps \g_\theta \to \g_\theta$ in $\X$ (an $\EVI_\lambda$-gradient flow is continuous up to $t=0$).
By lower semicontinuity of $\V$ this yields
\[
\V(\g_\theta) \leq \liminf_{\eps \downarrow 0} \V(\g^ \eps_\theta) \leq (1-\theta)\V(\g_0)+\theta \V(\g_1) -\frac\lambda 2 \theta(1-\theta)\sfd^2(\g_0,\g_1),
\]
whence the conclusion.
\end{proof}


\section{Derivative of the cost}
\label{sec:derivative_cost}

As a main application of the $\Gamma$-convergence results contained in Theorem \ref{t:gconv1} and Corollary \ref{c:gc} (and, in a wider sense, of their strategy of proof), in this section we investigate the dependence of the optimal value of the Schr\"odinger problem \eqref{eq:schrodinger_pb} on the regularization parameter $\eps$, focusing in particular on the regularity as a function of $\eps$ and on the behaviour in the small-time regime.
More precisely, and denoting
\[
\cost_\eps(x,y) := \inf_{(\g_t) \,:\, x \leadsto y}\Big\{\mathcal{A}(\g) + \eps^2\mathcal{I}(\g)\Big\}, \qquad \forall \eps \geq 0
\]
the optimal entropic cost, we show that $\eps \mapsto \cost_\eps(x,y)$ is (locally) absolutely continuous and admits explicit left and right derivatives in a pointwise sense. The strategy of proof follows an interpolation argument due to De Giorgi.
Moreover, since $\cost_\eps(x,y) \to \cost_0(x,y)$ as $\eps \downarrow 0$ by Corollary \ref{c:gc}, we aim at measuring the error $\cost_\eps(x,y) - \cost_0(x,y)$ and studying the minimizers of the unperturbed problem $\cost_0(x,y)$ selected by $\Gamma$-convergence.
Since we focus here on the dependence on $\eps$ we will assume throughout the whole Section~\ref{sec:derivative_cost} and without further mention the well-posedness of the $\eps$-Schr\"odinger problem:
\begin{hyp}
Fix $x,y \in \X$ and suppose that for some (hence for any, by Proposition \ref{prop:schro_solvable_iff_finite_entropy}) $\eps > 0$ the Schr\"odinger problem \eqref{eq:schrodinger_pb} admits at least one minimizer, in other words the infimum is attained in the definition of $\cost_\eps(x,y)$.
 \end{hyp}
 \noindent
We accordingly denote the set of $\eps$-minimizers as
\[
\Lambda_\eps(x,y) := \Big\{\omega \in AC^2([0,1],\X) \,:\, \omega_0 = x,\,\omega_1 = y
\qtext{and}\mathcal{A}_\eps(\omega) = \cost_\eps(x,y) \Big\}.
\]
Let us start the analysis with a preliminary monotonicity statement for the Fisher information and the entropic cost, which generalizes \cite[Lemma 3.3]{ConTam19}.

\begin{lem}\label{lem:monotone-fisher}
With the same assumptions and notations as in Setting \ref{setting} and for any $0 \leq \eps_1 < \eps_2 < \infty$ there holds
\[
\inf_{\Lambda_{\eps_1}(x,y)}\mathcal{I} \geq \sup_{\Lambda_{\eps_2}(x,y)}\mathcal{I},
\]
with possibly $\inf_{\Lambda_0(x,y)}\mathcal{I} = +\infty$.
Moreover, $\eps \mapsto \cost_\eps(x,y)$ is monotone non-decreasing on $[0,\infty)$.
\end{lem}

\begin{proof}
Let $\eps_1,\eps_2$ as in the statement and choose $\omega^i \in \Lambda_{\eps_i}(x,y)$ for $i=1,2$, so that by optimality
\[
\begin{split}
\mathcal{A}(\omega^1) + \eps_1^2\mathcal{I}(\omega^1) & \leq \mathcal{A}(\omega^2) + \eps_1^2\mathcal{I}(\omega^2), \\
\mathcal{A}(\omega^2) + \eps_2^2\mathcal{I}(\omega^2) & \leq \mathcal{A}(\omega^1) + \eps_2^2\mathcal{I}(\omega^1).
\end{split}
\]
Summing these inequalities and dividing by $\eps_2^2-\eps_1^2$ we obtain $\mathcal{I}(\omega^1) \geq \mathcal{I}(\omega^2)$, and since $\omega^1 \in \Lambda_{\eps_1}$ and $\omega^2 \in \Lambda_{\eps_2}$ are arbitrary the desired conclusion follows.
As regards the last part of the statement, it is sufficient to note that since $\omega^i$ are minimizers of their respective problems and $\eps_1 < \eps_2$,
$$
\cost_{\eps_1}(x,y) = \mathcal{A}(\omega^1) + \eps_1^2\mathcal{I}(\omega^1)
\leq
\mathcal{A}(\omega^2) + \eps_1^2\mathcal{I}(\omega^2)
\leq 
\mathcal{A}(\omega^2) + \eps_2^2\mathcal{I}(\omega^2)
=
\cost_{\eps_2}(x,y).
$$
\end{proof}
Let us then extend Theorem \ref{t:gconv1} and Corollary \ref{c:gc} from $\eps=0$ to any $\eps \geq 0$.
\begin{prop}\label{prop:extended_gamma_convergence}
With the same assumptions and notations as in Setting \ref{setting} and under the additional Assumption \ref{hyp1}, for any $\eps > 0$ there holds
\begin{equation}
\label{eq:extended-gammaconv}
\Gamma-\lim_{\eps' \to \eps}\Big\{ \mathcal{A}_{\eps'} +\iota_{01}\Big\} = \mathcal{A}_\eps + \iota_{01}
\end{equation}
for the pointwise-in-time $\sigma$-topology and
\[
\lim_{\eps' \to \eps}\cost_{\eps'}(x,y) = \cost_\eps(x,y).
\]
Moreover, for any $\eps_k \to \eps$ and any minimizer $\omega^k \in \Lambda_{\eps_k}(x,y)$, there exists a minimizer $\omega \in \Lambda_\eps(x,y)$ such that, up to a subsequence,
\[
\omega_t^k \stackrel{\sigma}{\to} \omega_t, \qquad \forall t \in [0,1]
\]
as $k \to \infty$.
\end{prop}

\begin{proof}
It is sufficient to prove \eqref{eq:extended-gammaconv}, as the other properties follow by a verbatim application of the arguments in the proof of Corollary \ref{c:gc}.

Fix $\eps$ and take $\eps'\to\eps$.
The $\Gamma-\limsup$ inequality is trivial:
if $\g^\eps$ is such that the right-hand side of \eqref{eq:extended-gammaconv} is finite (otherwise there is nothing to prove), then the constant sequence $\g^{\eps'} \equiv \g^\eps$ is an admissible recovery sequence.
For the $\Gamma-\liminf$ inequality, note that the kinetic action $\mathcal{A}$ and the Fisher information $\mathcal{I}$ are lower semicontinuous w.r.t.\ pointwise-in-time $\sigma$-convergence (see the proof of Proposition \ref{prop:schro_solvable_iff_finite_entropy}), and clearly so is the convex indicator.
Hence for any $\g^{\eps'}$ converging to $\g^\eps$ for the pointwise-in-time $\sigma$-topology it holds
\[
\begin{split}
\mathcal{A}_\eps(\g^\eps)+\iota_{01}(\g^\eps) 
& \leq \liminf_{\eps' \to \eps} \Big\{\mathcal{A}_\eps(\g^{\eps'})+\iota_{01}(\g^{\eps'})\Big\} \\
& = \liminf_{\eps' \to \eps} \Big\{\mathcal A(\g^{\eps'})+\eps^2\mathcal{I}(\g^{\eps'})+\iota_{01}(\g^{\eps'})\Big\} \\
& = \liminf_{\eps' \to \eps} \Big\{\mathcal A(\g^{\eps'})+(\eps')^2\mathcal{I}(\g^{\eps'})+\iota_{01}(\g^{\eps'})\Big\} \\
& = \liminf_{\eps' \to \eps} \Big\{ \mathcal{A}_{\eps'}(\g^{\eps'})+\iota_{01}(\g^{\eps'})\Big\}.
\end{split}
\]
\end{proof}

As an immediate consequence of this result we deduce the following

\begin{lem}\label{lem:continuity}
With the same assumptions and notations as in Setting \ref{setting} and under Assumption \ref{hyp1}, the function $\eps \mapsto \cost_\eps(x,y)$ is continuous on $[0,\infty)$.

Moreover, if $\eps \mapsto \omega^\eps$ is a continuous (w.r.t.\ the pointwise-in-time $\sigma$-topology) selection of minimizers, then $\eps \mapsto \mathcal{A}(\omega^\eps)$ and $\eps \mapsto \mathcal{I}(\omega^\eps)$ are also continuous, on $[0,\infty)$ and $(0,\infty)$ respectively.
\end{lem}

Note that if the minimizers are unique, then $\eps \mapsto \o^\eps$ is automatically continuous w.r.t.\ the pointwise-in-time $\sigma$-topology, simply by Proposition \ref{prop:extended_gamma_convergence}, as any sequence of minimizers admits a subsequence converging to a minimizer and the limit is in fact unique.
Also, the continuity of the Fisher information can be strengthened up to $\eps=0$, see later on Theorem~\ref{thm:taylor}.

\begin{proof}
The continuity of $\cost_\eps(x,y)$ for $\eps>0$ is granted by Proposition \ref{prop:extended_gamma_convergence}, while continuity at $\eps=0$ has already been proved in Corollary \ref{c:gc}.

As regards the kinetic energy $\mathcal{A}$ and the Fisher information $\mathcal{I}$, recall that they are both lower semicontinuous in $[0,\infty)$ w.r.t.\ the pointwise-in-time $\sigma$-topology, as already discussed in the proof of Proposition \ref{prop:schro_solvable_iff_finite_entropy}. Thus, if $\eps \mapsto \omega^\eps$ is as in the statement, we are left to prove that $\eps \mapsto \mathcal{A}(\omega^\eps)$ and $\eps \mapsto \mathcal{I}(\omega^\eps)$ are upper semicontinuous. To this aim, it is sufficient to observe that
\[
\begin{split}
\limsup_{\eps' \to \eps}\mathcal{A}(\omega^{\eps'}) 
& = \limsup_{\eps' \to \eps}\Big\{\cost_{\eps'}(x,y) - (\eps')^2 \mathcal{I}(\omega^{\eps'})\Big\} \\
& \leq \limsup_{\eps' \to \eps}\cost_{\eps'}(x,y) - \liminf_{\eps' \to \eps} (\eps')^2\mathcal{I}(\omega^{\eps'}) 
\\
& \leq \cost_\eps(x,y) - \eps^2\mathcal{I}(\omega^\eps) = \mathcal{A}(\omega^\eps),
\end{split}
\]
where the last inequality holds by the continuity of $\eps \mapsto \cost_\eps(x,y)$ and the lower semicontinuity of $\eps \mapsto \mathcal{I}(\omega^\eps)$.
Thus $\eps \mapsto \mathcal{A}(\omega^\eps)$ is upper semicontinuous in $[0,\infty)$.
Interchanging $\mathcal{A}$ and $\mathcal{I}$ and writing now $\mathcal{I} = \frac{1}{\eps^2}(\cost_\eps-\mathcal{A})$, the same argument shows that $\eps \mapsto \mathcal{I}(\omega^\eps)$ is upper semicontinuous in $(0,\infty)$ (continuity at $\eps=0$ will require a special treatment later).
\end{proof}
We have now all the ingredients to discuss the regularity of the cost $\cost_\eps(x,y)$ as a function of the noise parameter $\eps$ and explicitly compute its left and right derivatives.

\begin{prop}\label{prop:derivative}
With the same assumptions and notations as in Setting \ref{setting} and if Assumption \ref{hyp1} holds, the map $\eps \mapsto \cost_\eps(x,y)$ is $AC_{loc}([0,\infty))$, left and right differentiable everywhere in $(0,\infty)$ and, for any $\eps > 0$, the left and right derivatives are given by
\begin{equation}
\label{eq:derivative}
\frac{\d^-}{\d\eps}\cost_\eps(x,y)= 2\eps\max_{\Lambda_\eps(x,y)}\mathcal{I}, 
\qquad 
\frac{\d^+}{\d\eps}\cost_\eps(x,y) = 2\eps\min_{\Lambda_\eps(x,y)}\mathcal{I}
\end{equation}
respectively, and the former (resp.\ latter) is left (resp.\ right) continuous.
\end{prop}
\noindent
Note that it is part of our statement that the maximum and the minimum are attained.

\begin{rmk}{\rm
Heuristically, \eqref{eq:derivative} is nothing but the envelope theorem.
Indeed, if $\eps \mapsto \cost_\eps(x,y)$ were differentiable, then its derivative would be given by $\partial_\eps\mathcal{A}_\eps = 2\eps\mathcal{I}$ evaluated at any critical point, i.e.\ at any $\omega^\eps\in\Lambda_\eps(x,y)$.
However, since we do not know in our general metric framework that Schr\"odinger problem has a unique solution, we are not able to prove pointwise differentiability as in \cite{ConTam19} and we have to face the possibility of a gap between the left and right derivatives.
In any case, for a.e.\ $\eps>0$ this gap is zero, because $\eps \mapsto \cost_\eps(x,y)$ is locally absolutely continuous in $(0,\infty)$ and therefore a.e.\ differentiable.
This means that, up to a negligible set of temperatures, the left and right derivatives match and $\mathcal{I}$ is constant on $\Lambda_\eps(x,y)$. 
All these facts as well as the strategy of proof closely follow a variational interpolation argument which adapts the envelope theorem to the non-smooth setting and, to the best of our knowledge, dates back to De Giorgi.

If for whatever reason the Schr\"odinger problem \eqref{eq:schrodinger_pb} were uniquely solvable (which is in particular true for the classic Schr\"odinger problem, as proved in \cite[Theorem 4.2]{GigTam20}), then the left and right derivatives would be trivially equal and Lemma \ref{lem:continuity} would give that $\eps \mapsto \cost_\eps(x,y)$ is actually $C^1((0,\infty))$.

}\fr 
\end{rmk}

\begin{proof}
The continuity of $\eps \mapsto \cost_\eps(x,y)$ follows by Lemma \ref{lem:continuity}, so let us focus on left and right differentiability/continuity and local absolute continuity. 

\smallskip

\noindent{\bf Right differentiability}.
Fix $\eps>0$, let $\delta>0$, and choose $\omega^\eps \in \Lambda_\eps(x,y)$, $\omega^{\eps+\delta} \in \Lambda_{\eps+\delta}(x,y)$. Then write
\begin{equation}\label{eq:difference-quotient}
\begin{split}
\frac{\cost_{\eps+\delta}(x,y) - \cost_\eps(x,y)}{\delta} & = \frac{\mathcal{A}_{\eps+\delta}(\omega^{\eps+\delta})-\mathcal{A}_\eps(\omega^\eps)}{\delta} \\
& = \frac{\mathcal{A}_{\eps+\delta}(\omega^{\eps+\delta})-\mathcal{A}_{\eps+\delta}(\omega^\eps)}{\delta} + \frac{\mathcal{A}_{\eps+\delta}(\omega^\eps)-\mathcal{A}_\eps(\omega^\eps)}{\delta}
\end{split}
\end{equation}
and note that the second term on the right-hand side can be rewritten as
\[
\mathcal{A}_{\eps+\delta}(\omega^\eps)-\mathcal{A}_\eps(\omega^\eps) = (2\eps\delta + \delta^2)\mathcal{I}(\omega^\eps).
\]
The first one is non-positive by optimality of $\omega^{\eps+\delta}$ for $\mathcal{A}_{\eps+\delta}$, hence we obtain
\[
\limsup_{\delta \downarrow 0}\frac{\cost_{\eps+\delta}(x,y) - \cost_\eps(x,y)}{\delta} \leq \limsup_{\delta \downarrow 0}\,(2\eps + \delta)\mathcal{I}(\omega^\eps) = 2\eps\mathcal{I}(\omega^\eps).
\]
As this inequality holds for any $\omega^\eps \in \Lambda_\eps(x,y)$, we infer that
\begin{equation}\label{eq:limsup}
\limsup_{\delta \downarrow 0}\frac{\cost_{\eps+\delta}(x,y) - \cost_\eps(x,y)}{\delta} \leq 2\eps\inf_{\Lambda_\eps(x,y)}\mathcal{I}.
\end{equation}
On the other hand we can also write
\begin{equation}\label{eq:difference-quotient2}
\begin{split}
\frac{\cost_{\eps+\delta}(x,y) - \cost_\eps(x,y)}{\delta} & = \frac{\mathcal{A}_{\eps+\delta}(\omega^{\eps+\delta})-\mathcal{A}_\eps(\omega^\eps)}{\delta} \\
& = \frac{\mathcal{A}_{\eps+\delta}(\omega^{\eps+\delta})-\mathcal{A}_\eps(\omega^{\eps+\delta})}{\delta} + \frac{\mathcal{A}_\eps(\omega^{\eps+\delta})-\mathcal{A}_\eps(\omega^\eps)}{\delta}.
\end{split}
\end{equation}
Using now the optimality of $\omega^\eps$ for $\mathcal{A}_\eps$, we observe that the second term on the right-hand side is non-negative, whence
\[
\frac{\mathcal{A}_{\eps+\delta}(\omega^{\eps+\delta})-\mathcal{A}_\eps(\omega^\eps)}{\delta} \geq \frac{\mathcal{A}_{\eps+\delta}(\omega^{\eps+\delta})-\mathcal{A}_\eps(\omega^{\eps+\delta})}{\delta} = (2\eps + \delta)\mathcal{I}(\omega^{\eps+\delta}).
\]
For any sequence $\delta_n\downarrow 0$, Proposition \ref{prop:extended_gamma_convergence} guarantees (up to extraction of a subsequence if needed) that $\omega^{\eps+\delta_{n}} \to \overline{\omega}^\eps$ in the pointwise-in-time $\sigma$-topology for some $\overline{\omega}^\eps \in \Lambda_\eps(x,y)$.
By lower semicontinuity of $\mathcal I$ this implies
\[
\liminf_{n \to \infty}\frac{\cost_{\eps+\delta_n}(x,y) - \cost_\eps(x,y)}{\delta_n} \geq \liminf_{n \to \infty}(2\eps + \delta_n)\mathcal{I}(\omega^{\eps+\delta_n}) \geq 2\eps\mathcal{I}(\overline{\omega}^\eps) \geq 2\eps\inf_{\Lambda_\eps(x,y)}\mathcal{I},
\]
and together with \eqref{eq:limsup} this yields
\[
\exists \lim_{n \to \infty}\frac{\cost_{\eps+\delta_n}(x,y) - \cost_\eps(x,y)}{\delta_n} = 2\eps\mathcal{I}(\overline{\omega}^\eps) = 2\eps\inf_{\Lambda_\eps(x,y)}\mathcal{I}.
\]
As the right-hand side does not depend on the particular sequence $\delta_n\downarrow 0$ we conclude that
\[
\exists \lim_{\delta \downarrow 0}\frac{\cost_{\eps+\delta}(x,y) - \cost_\eps(x,y)}{\delta} = 2\eps\min_{\Lambda_\eps(x,y)}\mathcal{I},
\]
in particular $\mathcal{I}$ is minimized by any accumulation point $\overline\omega^\eps$ of $\{\omega^{\eps+\delta}\}_{\delta > 0}$. 

\noindent{\bf Left differentiability}.
The argument is very similar.
Indeed, if $\delta < 0$, then the first term on the right-hand side of \eqref{eq:difference-quotient} is non-negative and the second one can be handled in the same way.
Hence there holds
\[
\liminf_{\delta \uparrow 0}\frac{\cost_{\eps+\delta}(x,y) - \cost_\eps(x,y)}{\delta} \geq 2\eps\mathcal{I}(\omega^\eps),
\]
for any $\omega^\eps \in \Lambda_\eps(x,y)$, and therefore
\[
\liminf_{\delta \uparrow 0}\frac{\cost_{\eps+\delta}(x,y) - \cost_\eps(x,y)}{\delta} \geq 2\eps\sup_{\Lambda_\eps(x,y)}\mathcal{I}.
\]
Applying the same considerations to \eqref{eq:difference-quotient2} and following the same argument as above we retrieve the $\limsup$ inequality, first along \emph{some} subsequence $\delta_n \uparrow 0$ and then along \emph{any} $\delta \uparrow 0$.
Combining with the inequality above gives
\[
\exists \lim_{\delta \uparrow 0}\frac{\cost_{\eps+\delta}(x,y) - \cost_\eps(x,y)}{\delta} = 2\eps\max_{\Lambda_\eps(x,y)}\mathcal{I}, \qquad \forall \eps > 0,
\]
whence the pointwise left differentiability of $\eps \mapsto \cost_\eps(x,y)$.

\noindent{\bf Left and right continuity}.
In order to prove the right continuity of the right derivative of $\eps \mapsto \cost_\eps(x,y)$, note that on the one hand by Lemma \ref{lem:monotone-fisher} for any $\eps_n \downarrow \eps$ it holds
\[
\inf_{\Lambda_\eps(x,y)}\mathcal{I} \geq \limsup_{n \to \infty}\sup_{\Lambda_{\eps_n}(x,y)}\mathcal{I} \geq \limsup_{n \to \infty}\inf_{\Lambda_{\eps_n}(x,y)}\mathcal{I}.
\]
On the other hand, we can assume up to a subsequence if needed that
\[
\liminf_{n \to \infty}\inf_{\Lambda_{\eps_n}(x,y)}\mathcal{I} = \lim_{n \to \infty}\inf_{\Lambda_{\eps_n}(x,y)}\mathcal{I}.
\]
As shown in the proof of right differentiability, $\inf_{\Lambda_{\eps'}(x,y)}\mathcal{I}$ is attained for any $\eps' > 0$, hence in particular $\inf_{\Lambda_{\eps_n}(x,y)}\mathcal{I} = \mathcal{I}(\omega^n)$ for some $\omega^n \in \Lambda_{\eps_n}(x,y)$, for all $n$.
Up to extracting a further subsequence, by Proposition \ref{prop:extended_gamma_convergence} we can assume that $\omega^n \to \overline{\omega}^\eps$ w.r.t.\ the pointwise-in-time $\sigma$-topology for some $\overline{\omega}^\eps \in \Lambda_\eps(x,y)$, and moreover by Lemma \ref{lem:continuity}
\[ 
\lim_{n \to \infty}\mathcal{I}(\omega^n) = \mathcal{I}(\overline{\omega}) \geq \inf_{\Lambda_\eps(x,y)}\mathcal{I}.
\]
Putting all these inequalities together provides us with the right continuity of $\eps \mapsto \inf_{\Lambda_\eps(x,y)}\mathcal{I}$ and, a fortiori, of the right derivative. Left continuity for the left derivative follows along an analogous reasoning.

\noindent{\bf Local absolute continuity}.
Let $0 < \eps_1 < \eps_2 < \infty$ and, for any $0 < \delta < 1$, define
\[
f_\delta(\eps) := \frac{\cost_{\eps+\delta}(x,y)-\cost_\eps(x,y)}{\delta}.
\]
The monotonicity of $\eps \mapsto \cost_\eps(x,y)$ from Lemma \ref{lem:monotone-fisher} gives $f_\delta \geq 0$.
Arguing as in the very beginning of the proof of the right differentiability we see that $f_\delta(\eps) \leq (2\eps+1)\mathcal{I}(\omega^\eps)$ for any $\omega^\eps \in \Lambda_\eps(x,y)$, and by Lemma \ref{lem:monotone-fisher}
\[
f_\delta(\eps) \leq (2\eps_2+1)\sup_{\Lambda_{\eps_1}(x,y)}\mathcal{I} < \infty, \qquad \forall \eps \in (\eps_1,\eps_2].
\]
Hence $|f_\delta| \leq M$ uniformly in $\delta$ and $f_\delta$ converges pointwise to the right derivative of $\eps \mapsto \cost_\eps(x,y)$ as $\delta\downarrow 0$, whence by the dominated convergence theorem
\[
\int_{\eps_1}^{\eps_2}\frac{\d^+}{\d\eps}\cost_\eps(x,y)\,\d\eps = \lim_{\delta \downarrow 0}\int_{\eps_1}^{\eps_2}f_\delta(\eps)\,\d\eps.
\]
The right-hand side can be rewritten as
\[
\begin{split}
\lim_{\delta \downarrow 0}\int_{\eps_1}^{\eps_2}f_\delta(\eps)\,\d\eps & = \lim_{\delta \downarrow 0}\Big(\frac{1}{\delta}\int_{\eps_1}^{\eps_2}\cost_{\eps+\delta}(x,y)\,\d\eps - \frac{1}{\delta}\int_{\eps_1}^{\eps_2}\cost_\eps(x,y)\,\d\eps\Big) \\
& = \lim_{\delta \downarrow 0}\Big(\frac{1}{\delta}\int_{\eps_2}^{\eps_2+\delta}\cost_\eps(x,y)\,\d\eps - \frac{1}{\delta}\int_{\eps_1}^{\eps_1+\delta}\cost_\eps(x,y)\,\d\eps\Big) \\
& = \cost_{\eps_2}(x,y) - \cost_{\eps_1}(x,y),
\end{split}
\]
where the last equality holds by the Lebesgue differentiation theorem for the continuous function $\eps \mapsto \cost_\eps(x,y)$ (cf.\ Lemma \ref{lem:continuity}).
We have thus proved that the cost belongs to $AC_{loc}((0,\infty))$, since
\[
\cost_{\eps_2}(x,y) - \cost_{\eps_1}(x,y) = \int_{\eps_1}^{\eps_2}\frac{\d^+}{\d\eps}\cost_\eps(x,y)\,\d\eps,
\qquad \forall\, 0 < \eps_1 < \eps_2.
\]
For the full $AC_{loc}([0,\infty))$ regularity it is then sufficient to let $\eps_1 \downarrow 0$:
the left-hand side converges to $\cost_{\eps_2}(x,y) - \cost_0(x,y)$ by Lemma \ref{lem:continuity}, and by the monotonicity $\frac{\d^+}{\d\eps}\cost_\eps\geq 0$ the right-hand side also converges by monotone convergence.
\end{proof}
As a consequence we get
\begin{cor}
 The map $\eps\mapsto \cost_\eps(x,y)$ is locally semiconcave, and more precisely for any $\eps_0>0$ there holds
 $$
 \frac{\d^2}{\d\eps^2}\cost_\eps(x,y)\leq 2 \max\limits_{\Lambda_{\eps_0}(x,y)}\mathcal I<+\infty
 $$
 in the sense of distributions in $(\eps_0,+\infty)$.
 In particular $\eps\mapsto \cost_\eps(x,y)$ is twice differentiable almost everywhere.
\end{cor}
\begin{proof}
From Proposition~\ref{prop:derivative} we know that, being absolutely continuous, the map $\eps\mapsto \cost_\eps(x,y)$ is differentiable a.e., in particular the left and right derivatives agree and thus
$$
\mathcal I_\eps(x,y)
:=
\max_{\Lambda_\eps(x,y)}\mathcal{I}
=\min_{\Lambda_\eps(x,y)}\mathcal{I}
$$
for a.a. $\eps>0$ with
$$
\frac{\d}{\d\eps}\cost_\eps(x,y)=2\eps \mathcal I_\eps(x,y)
$$
in the sense of distributions.
From Lemma~\ref{lem:monotone-fisher} $\mathcal I_\eps$ is non-increasing in $\eps$, hence one can legitimately compute in the distributional sense
$$
\frac{\d^2}{\d\eps^2}\cost_\eps(x,y)=\frac{\d}{\d\eps}\left\{2\eps \mathcal I_\eps(x,y)\right\}
=2\mathcal I_\eps(x,y) + 2\eps\frac{\d}{\d\eps}\mathcal I_\eps(x,y)
\leq2\mathcal I_\eps(x,y)
$$
and the conclusion follows again by monotonicity for $\eps\geq \eps_0$.
The a.e. twice differentiability follows by Alexandrov's theorem.
\end{proof}

Relying on our previous auxiliary results and on Proposition \ref{prop:derivative}, we are finally in position of estimating the error $\cost_\eps(x,y) - \cost_0(x,y)$ with $o(\eps^2)$ precision.
We will also significantly refine Corollary \ref{c:gc} by proving that any accumulation point of any sequence of minimizers is not only optimal for the unperturbed problem $\cost_0(x,y)$, but also $\mathcal{I}$-minimizing among all competitors in $\Lambda_0(x,y)$.

\begin{theo}\label{thm:taylor}
With the same assumptions and notations as in Proposition \ref{prop:derivative}, if there exists $\omega^0 \in \Lambda_0(x,y)$ such that $\mathcal{I}(\omega^0) < \infty$, then the map $\eps \mapsto \cost_\eps(x,y)$ is right differentiable also at $\eps=0$ with
\[
\frac{\d^+}{\d\eps}\cost_\eps(x,y)\Big|_{\eps=0} = 0,
\]
the right derivative is right continuous for any $\eps \geq 0$, and
\begin{equation}\label{eq:taylor}
\cost_\eps(x,y) - \cost_0(x,y) = \eps^2\inf_{\Lambda_0(x,y)}\mathcal{I} + o(\eps^2).
\end{equation}
Moreover, for any $\eps_n \downarrow 0$ and any minimizer $\omega^n \in \Lambda_{\eps_n}(x,y)$ there exists $\omega^* \in \Lambda_0(x,y)$ such that (up to a subsequence) $\omega^n \to \omega^*$ for the pointwise-in-time $\sigma$-topology, and $\omega^*$ has minimal Fisher information in $\Lambda_0(x,y)$
\[
\mathcal{I}(\omega^*) = \min_{\Lambda_0(x,y)}\mathcal{I}.
\]
\end{theo}

\begin{proof}
The right differentiability of $\eps \mapsto \cost_\eps(x,y)$ at $\eps=0$ follows by the same argument carried out in Proposition \ref{prop:derivative}. Indeed, given $\omega^0$ as in the statement, by \eqref{eq:difference-quotient} with $\eps=0$ it holds
\[
\limsup_{\delta \downarrow 0}\frac{\cost_\delta(x,y) - \cost_0(x,y)}{\delta} \leq \limsup_{\delta \downarrow 0} \delta\mathcal{I}(\omega^0) = 0.
\]
The liminf inequality is straightforward, since $\mathcal{I} \geq 0$ and thus by \eqref{eq:difference-quotient2} with $\eps=0$
\[
\liminf_{\delta \downarrow 0}\frac{\cost_\delta(x,y) - \cost_0(x,y)}{\delta} \geq \liminf_{\delta \downarrow 0} \delta\mathcal{I}(\omega^\delta) \geq 0
\]
for any $\omega^\delta \in \Lambda_\delta(x,y)$.
This also shows that the right derivative vanishes at $\eps=0$. 

As regards the right continuity of the right derivative, the case $\eps>0$ has already been discussed in Proposition \ref{prop:derivative}.
For $\eps=0$ the same strategy still works, with the only minor difference that we cannot rely on Lemma \ref{lem:continuity} anymore.
Nonetheless, if $\omega^n \in \Lambda_{\eps_n}(x,y)$ is as in Proposition \ref{prop:derivative}, $\overline{\omega} \in \Lambda_0(x,y)$ and $\omega^n \to \overline{\omega}$ for the pointwise-in-time $\sigma$-topology (the existence of such $\overline{\omega}$ is granted by Corollary \ref{c:gc}) it is still true that
\[
\liminf_{n \to \infty}\mathcal{I}(\omega^n) \geq \mathcal{I}(\overline{\omega}),
\]
simply by lower semicontinuity of $\mathcal{I}$.
With this single change in the proof we deduce that $\eps \mapsto \inf_{\Lambda_\eps(x,y)}\mathcal{I}$ is right continuous and finite also at $\eps=0$, thanks to the present assumptions, and so is the right derivative of the cost due to $\eps\inf_{\Lambda_\eps(x,y)}\mathcal{I} \to 0$ as $\eps \downarrow 0$.

The last part of the statement is a slight modification of these lines of thought. Indeed, given any sequence $\eps_n \downarrow 0$ and $\omega^n \in \Lambda_{\eps_n}(x,y)$, the existence of $\omega^* \in \Lambda_0(x,y)$ such that, up to subsequences, $\omega^n \to \omega^*$ is ensured by Corollary \ref{c:gc}.
The fact that $\omega^*$ has minimal Fisher information among all elements in $\Lambda_0(x,y)$ follows from
\[
\inf_{\Lambda_0(x,y)}\mathcal{I} \geq \limsup_{n \to \infty}\sup_{\Lambda_{\eps_n}(x,y)}\mathcal{I} \geq \limsup_{n \to \infty}\mathcal{I}(\omega^n) \geq \liminf_{n \to \infty}\mathcal{I}(\omega^n) \geq \mathcal{I}(\omega^*) \geq \inf_{\Lambda_0(x,y)}\mathcal{I},
\]
where we used once again Lemma \ref{lem:monotone-fisher} and the lower semicontinuity of $\mathcal{I}$.

Thus, it only remains to establish \eqref{eq:taylor}.
As $\eps \mapsto \cost_\eps(x,y)$ belongs to $AC_{loc}([0,\infty))$ and the right derivative coincides a.e.\ with the full derivative, \eqref{eq:derivative} and the fundamental theorem of calculus yield
\begin{equation}\label{lastissue}
\cost_\eps(x,y) - \cost_0(x,y) 
= 2\int_0^\eps s\inf_{\Lambda_s(x,y)}\mathcal{I}\,\d s 
\leq 2\int_0^\eps s\inf_{\Lambda_0(x,y)}\mathcal{I}\,\d s
= \eps^2 \inf_{\Lambda_0(x,y)}\mathcal{I}.
\end{equation}
Here we used the monotonicity of the Fisher information from Lemma~\ref{lem:monotone-fisher} in the middle inequality.
By the same monotonicity and the right continuity at $\eps=0$ of $\eps \mapsto \inf_{\Lambda_\eps(x,y)}\mathcal{I}$ we also deduce that
\[
\begin{split}
\cost_\eps(x,y) - \cost_0(x,y) & \geq 2\int_0^\eps s\inf_{\Lambda_\eps(x,y)}\mathcal{I}\,\d s = \eps^2 \inf_{\Lambda_\eps(x,y)}\mathcal{I} \\
& = \eps^2 \inf_{\Lambda_0(x,y)}\mathcal{I} + \eps^2 \Big(\inf_{\Lambda_\eps(x,y)}\mathcal{I} - \inf_{\Lambda_0(x,y)}\mathcal{I}\Big) \\
& = \eps^2 \inf_{\Lambda_0(x,y)}\mathcal{I} + o(\eps^2)
\end{split}.
\]
Combining this lower bound with the previous upper one entails \eqref{eq:taylor}.
\end{proof}

\begin{rmk}{\rm
It is worth stressing that the upper bound \eqref{lastissue} on $\cost_\eps(x,y) - \cost_0(x,y)$ is not asymptotic, but pointwise, in the sense that it holds for all $\eps>0$. Since in addition $\cost_\eps(x,y) - \cost_0(x,y)$ is always non-negative, we can rewrite \eqref{eq:taylor} in the following quantitative way: \begin{equation}\label{lastissue1}
	|\cost_\eps(x,y) - \cost_0(x,y) |
	\leq \eps^2 \inf_{\Lambda_0(x,y)}\mathcal{I}.
\end{equation}
A possible way to further improve \eqref{eq:taylor} would rely on a refined analysis of $\eps \mapsto \inf_{\Lambda_\eps(x,y)}\mathcal{I}$, its derivative (which exists a.e.\ by monotonicity), and possibly absolute continuity.
}\fr 
\end{rmk}

\begin{rmk}{\rm
The $\mathcal{I}$-minimizing property of the accumulation point $\omega^*$ is not specific of the choice $\eps=0$, but of the particular ``backward'' direction of the sequence $\eps_n\downarrow$.
Repeating the argument in the proof of Theorem \ref{thm:taylor} it is indeed not difficult to check that, given any $\eps>0$, a sequence $\eps_n \downarrow \eps$, and $\omega^n \in \Lambda_{\eps_n}(x,y)$ there exists $\omega^\eps \in \Lambda_\eps(x,y)$ such that, up to a subsequence, $\omega^n \to \omega^\eps$ for the pointwise-in-time $\sigma$-topology and
\[
\mathcal{I}(\omega^\eps) = \inf_{\Lambda_\eps(x,y)}\mathcal{I}.
\]
In a symmetric fashion, a closer look into the proof of Proposition \ref{prop:derivative} suggests that an opposite behaviour appears in the ``forward'' direction.
More precisely, if $\eps_n \uparrow \eps$ instead of $\eps_n \downarrow \eps$, then any accumulation point $\omega^\eps$ of $(\omega^n)$ is such that
\[
\mathcal{I}(\omega^\eps) = \sup_{\Lambda_\eps(x,y)}\mathcal{I}.
\]
However, the ``backward'' direction and the case $\eps=0$ are usually more interesting, because of the connection with the unperturbed problem $\cost_0(x,y)$ for which there might be multiple solutions even if the Schr\"odinger problem \eqref{eq:schrodinger_pb} has a unique minimizer for all $\eps>0$.
It is therefore natural to look for the (properties of the) solutions selected via Schr\"odinger regularization.
}\fr 
\end{rmk}

\begin{rmk}
{\rm 
In the present work, the metric Schr\"odinger problem has been essentially studied in the $\eps \downarrow 0$ regime, where it connects to the geodesic problem. In a forthcoming paper, the second author \cite{Tamanini+} will investigate the problem at $\eps$ fixed and in the long-time regime $\eps \to \infty$, with focus on ergodicity and energy estimates.
}\fr 
\end{rmk}


\section{Examples}
\label{s: exam}
In this section we collect several and heterogeneous situations where our abstract approach (in particular Theorem \ref{t:gconv1}, Corollary \ref{c:gc}, and Theorem \ref{thm:taylor}) applies.
We shall also comment the novelty of the results thus obtained in comparison with the existing literature.
In this perspective, it is worth discussing first in more detail the role played by Assumption \ref{hyp1} so far, singling out when the sequential lower semicontinuity of $|\partial\V|$ w.r.t.\ $\sigma$ is needed and when it is not:
\begin{itemize}
\item 
to prove existence of a solution to the Schr\"odinger problem \eqref{eq:schrodinger_pb} (cf.\ Proposition \ref{prop:schro_solvable_iff_finite_entropy}) it is crucial, in order to apply the direct method of the calculus of variations;
\item 
in Theorem \ref{t:gconv1} and Corollary \ref{c:gc} it is not used;
\item
unlike Corollary \ref{c:gc}, in Proposition \ref{prop:extended_gamma_convergence} it is needed for the $\Gamma$-liminf inequality and so is in Lemma \ref{lem:continuity};
\item
Proposition \ref{prop:derivative} relies on Proposition \ref{prop:extended_gamma_convergence} and Lemma \ref{lem:continuity}, hence it is implicitly used;
\item in Theorem \ref{thm:taylor} the continuity of $\eps \mapsto \inf_{\Lambda_\eps(x,y)}\mathcal{I}$ at $\eps=0$ requires the lower semicontinuity of $|\partial\V|$ and also Proposition \ref{prop:derivative} is used in the proof of \eqref{eq:taylor}; hence the lower semicontinuity of $|\partial\V|$ is really needed.
\end{itemize}
This means that if one is able to show the solvability of the Schr\"odinger problem \eqref{eq:schrodinger_pb} by means other than those used in Proposition \ref{prop:schro_solvable_iff_finite_entropy}, then Theorem \ref{t:gconv1} and Corollary \ref{c:gc} are still valid under the following weaker hypothesis.

\begin{hyp} \label{hyp2}
There exists a Hausdorff topology $\sigma$ on $\X$ such that $\sfd$-bounded sequences contain $\sigma$-converging subsequences. Moreover, the distance $\sfd$ is sequentially lower semicontinuous w.r.t.\ $\sigma$. 
\end{hyp}
\noindent
Theorem \ref{thm:taylor}, instead, requires the full validity of Assumption \ref{hyp1}.
\\

Let us now discuss some specific instances where our results apply.

\subsection{Hadamard spaces}

As a first example, we consider as an entropy functional the squared distance (to an arbitrary point), whose convexity is ensured by a non-positive sectional curvature assumption.

\subsubsection*{Setting}

\begin{itemize}\item Let $(\X,\sfd)$ be a complete and separable $\CAT(0)$ space (i.e.\ a separable Hadamard space) that satisfies a (rather mild) geometric $\overline{Q_4}$ condition \cite{kirk14}.
\item 
Let $\V(\cdot):=\frac 1 2 \sfd^2(x_0,\cdot)$, where $x_0\in \X$ is a fixed point.
\item
We employ the \emph{half-space topology} \cite{kakavandi2013} as $\sigma$.
The corresponding convergence is known as the $\Delta$-convergence \cite{lim76}. If $\X$ is locally compact, it coincides with the strong metric convergence. \end{itemize}

\subsubsection*{Applicability of our results}
\begin{itemize}\item
The example fits into Setting \ref{setting}. Indeed, by a basic property of $\CAT(0)$ spaces, $\V$ is a (continuous) $1$-convex functional; consequently, \cite[Theorem 3.14]{MuratoriSavare20} provides existence of an $\EVI_1$-gradient flow of $\V$ starting from any $x\in\X$. Note that $\V$ is always finite. 

\item Assumption \ref{hyp1} is also true. 
Indeed, $\sfd$-bounded sequences contain $\Delta$-converging subsequences \cite{lim76, kakavandi2013}. 
Bounded closed convex sets (in particular, balls) are $\Delta$-closed \cite{kirk14}, which easily implies that $\sfd$ is $\Delta$-lower semicontinuous.
Moreover, it is easy to see from \eqref{eq:local=global} that $|\partial\V(x)|= \sfd(x_0,x)$, thus the slope is $\Delta$-lower semicontinuous too. 
\end{itemize}

Hence, in this framework, all our results are applicable.

\subsection{The Boltzmann-Shannon relative entropy}\label{subsec:RCD}

Let us consider now the Boltzmann-Shannon relative entropy on the Wasserstein space built over a (locally compact) $\RCD$ space:

\subsubsection*{Setting}

\begin{itemize} 
\item
Let $M$ be a complete, separable, and locally compact length space endowed with a Radon measure $\m$, and assume that it is an $\RCD(K,\infty)$ space \cite{AmbrosioGigliSavare11-2} for some $K\in\R$.
We consider the $2$-Wasserstein space $\X := \mathcal P_2(M)$ over $M$, namely the space of probability measures with finite second moments, equipped with $\sfd:=W_2$, the $2$-Wasserstein distance.

\item
The Boltzmann-Shannon relative entropy $\V$ on $\X$ is defined as
\[
\V(\mu) := \left\{\begin{array}{ll}
\displaystyle{\int_M \rho\log(\rho)\,\d\m} & \qquad \text{ if } \mu = \rho\m,\\
+\infty & \qquad \text{ if } \mu \not\ll {\m}.
\end{array}\right.
\]
\item 
For the $\sigma$ topology we choose the metric topology of $(\X, W_p)$, where $W_p$ is the $p$-Wasserstein istance, $1\leq p<2$ \footnote{$(\X,W_2)$ is not locally compact unless $M$ is compact, so that in general the metric topology of $(\X,W_2)$ is not an admissible candidate for $\sigma$.}
\end{itemize}

\noindent With this choice of $(\X,\sfd)$ and $\V$ we recover the dynamical formulation of the ``classical'' Schr\"odinger problem \cite{Leonard14}. Indeed, taking into account the equivalence between $W_2$-absolutely continuous curves and distributional solutions of the continuity equation (see \cite{GigliHan13}) and the fact that the slope $|\partial\V|^2$ coincides with the Fisher information \cite[Theorem 9.3]{AmbrosioGigliSavare11}, \eqref{eq:schrodinger_pb} reads as
\[
\inf \bigg\{ \frac{1}{2}\iint_0^1 |v_t|^2\rho_t\,\d t\d\m + \frac{\eps^2}{2}\iint_0^1 |\nabla\log\rho_t|^2\rho_t\,\d t\d\m \bigg\},
\]
where the infimum runs over all couples $(\mu_t,v_t)$, $\mu_t = \rho_t\m$, solving the continuity equation $$\partial_t\mu_t + {\rm div}(v_t\mu_t) = 0$$ with the constraint $\mu_0=\mu$ and $\mu_1=\nu$.

\subsubsection*{Applicability of our results}

In order to see that our abstract metric results hold for this specific example, let us check separately the validity of Setting \ref{setting} and Assumption \ref{hyp2}.

\begin{itemize}  
\item Setting \ref{setting} fully holds. Indeed, $(\X,\sfd)$ is a complete and separable metric space \cite{Bolley8}. Moreover, by \cite[Theorem 4.24]{Sturm06I} there exist $C>0$, $x \in M$ such that $\int_M e^{-C\sfd^2(\cdot,x)}\d\m < \infty$. Consequently, $\V$ can be equivalently rewritten as
\[
\V(\mu) = \underbrace{\int_M \tilde{\rho}\log(\tilde{\rho})\,\d\tilde{\m}}_{\geq 0} - C\int_M \sfd^2(\cdot,x)\,\d\mu - \log Z,
\]
where $\tilde{\rho}$ is the Radon-Nikodym derivative of $\mu$ w.r.t.\ $\tilde{\m}$, with the normalization
\[
Z := \int_M e^{-C\sfd^2(\cdot,x)}\d\m, \qquad \tilde{\m} := \frac 1Z e^{-C\sfd^2(\cdot,x)}\m.
\]
From this very definition, it is easy to see that $\V$ is a proper lower semicontinuous functional, bounded from below on $W_2$-bounded sets. Finally, by (one of the equivalent) definition of $\RCD$ spaces, cf.\ \cite[Theorem 5.1]{AmbrosioGigliSavare11-2}, for any $\mu \in \X$ there exists an $\EVI_K$-gradient flow of $\V$ starting from it (in particular, this implies that $\V$ has a dense domain). 

\item Assumption \ref{hyp2} also holds. Indeed, $W_2$-bounded sequences in $\X$ are uniformly tight (the second moments are uniformly bounded and the balls in $M$ are relatively compact, so that the claim follows from \cite[Remark 5.1.5]{AmbrosioGigliSavare08}) and thus relatively compact w.r.t.\ the narrow topology. Passing to a subsequence, we may assume that such a sequence is narrowly convergent. Leveraging on its tightness and applying the H\"older inequality, it is easy to deduce that the $p$-th moments w.r.t. any fixed reference point, $p<2$, converge to the corresponding limiting $p$-th moment. Hence, this sequence is $\sigma$-converging. Moreover, $W_2$ is lower semicontinuous w.r.t.\ narrow convergence of measures \cite[Proposition 3.5]{AmbrosioGigli11} and thus w.r.t. $W_p$-convergence, $p<2$. 
\end{itemize}

Therefore, given any $\mu,\nu \in \X$ for which the dynamical Schr\"odinger problem \eqref{eq:schrodinger_pb} is solvable, the $\Gamma$-convergence results of Section \ref{sec:small_noise_convexity} are fully applicable. This is for instance the case if $\mu,\nu \ll \m$ have bounded densities and supports (in \cite{GigTam18, GigTam20} this is proved for $\RCD^*(K,N)$ spaces, $N<\infty$, but the argument can be adapted to locally compact $\RCD(K,\infty)$ spaces thanks to the existence of ``good'' cut-off functions \cite{Mondino-Naber14}).

The more demanding Assumption \ref{hyp1} is satisfied for example if
\begin{itemize}
\item 
either $M$ is a convex domain in $\R^d$; see \cite[Lemma 2.4]{GST09} for a proof of the narrow (and hence $W_p$-) lower semicontinuity of $|\partial\V|$).
\item
or under the assumption that $M$ is compact (e.g.\ the torus, the sphere or any convex closed bounded subset of a smooth weighted Riemannian manifold). In this case we can even choose the topology $\sigma$ to be the strong one induced by $W_2$, and $|\partial\V|$ is lower semicontinuous by Remark~\ref{rmk18}.
\end{itemize}
In these two situations all our abstract results are applicable.

\subsubsection*{Novelty and related literature} 

A thorough study of the ``classical'' Schr\"odinger problem and its equivalent formulations (at the static, dual, and dynamical levels) has been carried out by the second author in \cite{GigTam20}, but in the more restrictive framework of $\RCD^*(K,N)$ spaces, and only for $\eps$ fixed. The behaviour of the (unique) minimizers as $\eps \downarrow 0$ was instead studied in \cite[Proposition 5.1]{GigTam18}, again only in $\RCD^*(K,N)$ spaces, but the $\Gamma$-convergence of the corresponding variational problems was not investigated.  Hence Theorem \ref{t:gconv1} and Corollary \ref{c:gc} are new in the $\RCD$ framework.

As regards the validity of the results of Section \ref{sec:derivative_cost} in the two situations described above, this partly extends the recent work \cite{ConTam19}, where an analogue of Theorem \ref{thm:taylor} is proved in the Riemannian setting.

\subsection{Internal energies and the R\'enyi entropy}\label{subsec:renyi}

As a next class of examples, we consider generalized entropy functionals (usually called \emph{internal energies}) on the Wasserstein space built over an $\RCD^*(0,N)$ space, $N < \infty$.
The setting is therefore the following.

\subsubsection*{Setting}

\begin{itemize} 
\item 
Let $(\X,\sfd) := (\mathcal P_2(M), W_2)$ be the $2$-Wasserstein space over $M$. 
The underlying space $M$ is assumed to be an $\RCD^*(0,N)$ space\footnote{The notion of an $\RCD^*(K,N)$ space was introduced in \cite{Gigli12};
for comparison between $\RCD$ and $\RCD^*$ conditions see \cite{BacherSturm10} and \cite{CavMil16}, in particular, these notions coincide when $\m(M)<\infty$ and are expected to always coincide.} with reference measure $\m$, hence in particular $M$ is complete and separable.
\item
The (generalized) entropy/internal energy $\V$ on $\X$ is defined as
\begin{equation}
\label{eq:internal_energy_U}
\V(\mu) := \int_M U(\rho)\,\d\m + U'(\infty)\mu^\perp(M), \qquad \textrm{if } \mu = \rho\m + \mu^\perp, \,\mu^\perp \perp \m
\end{equation}
where $U'(\infty) := \lim_{r \to \infty} U'(r)$.
The function $U : [0,\infty) \to \R$ is assumed continuous and convex, with $U(0)=0$ and $U'$ locally Lipschitz in $(0,\infty)$ satisfying McCann's \cite{McC97} condition\footnote{This means that the corresponding pressure function $P(r) := rU'(r) - U(r)$ is such that $P(0) := \lim_{r \downarrow 0}P(r) = 0$ and $r \mapsto r^{-1+1/N'}P(r)$ is non-decreasing or, equivalently, $r \mapsto r^{N'}U(r^{-N'})$  is convex and non-increasing on $(0,+\infty)$.} for some $N' \in [N,+\infty)$.
\item 
The topology $\sigma$ will be again the metric topology of $(\X, W_p)$, $1\leq p<2$.
\end{itemize}

Note that in the case $U$ is chosen equal to
\[
U_{N'}(r) := -N'(r^{1-1/N'}-r), \,\, N' \geq N \qquad \textrm{or} \qquad U_m(r) := \frac{1}{m-1}r^m, \,\, m \geq 1-\frac{1}{N}
\]
($U_{N'}$ being more linked to Lott-Sturm-Villani theory of curvature-dimension bounds, $U_m$ with the porous medium equation of power $m$), the famous R\'enyi entropy is recovered. Detailed discussions of the internal energies associated to non-linear diffusion semigroups and evolution variational inequalities in connection with curvature-dimension conditions can be found in \cite{AmbrosioMondinoSavare13} and in \cite[Chapters 16 and 17]{Villani09}.

\subsubsection*{Applicability of our results}

Let us verify that all the conditions in Setting \ref{setting} and Assumption \ref{hyp2} hold.

\begin{itemize}
\item We are within Setting \ref{setting}. Indeed, by the discussion carried out in the previous section and by the fact that $\RCD^*(K,N)$ spaces are in particular locally compact $\RCD(K,\infty)$ spaces, $\X$ is a complete and separable metric space and $\sigma$ is an admissible topology. Moreover, since $U(0)=0$, $M$ is locally compact and $U$ is continuous, it is clear that $\V$ is well defined and finite on all probability measures with bounded support, so that $\V$ is proper and has a dense domain in $\X$. Actually, $D(\V)$ is dense in energy in $\X$, i.e.\ for all $\mu \in \X$ there exist $\mu_n \in D(\V)$ with $W_2(\mu_n,\mu) \to 0$ and $\V(\mu_n) \to \V(\mu)$ as $n \to \infty$. By the properties of $U$ it is also easy to see that $\V$ is lower semicontinuous \cite[Theorem 30.6]{Villani09} and bounded from below on $W_2$-bounded sets. Finally, from \cite[Theorem 9.21]{AmbrosioMondinoSavare13} with $K=0$ (since $M$ is assumed to be an $\RCD^*(0,N)$ space) and the fact that $D(\V)$ is dense in energy in $\X$, we see that for all $\mu \in \X$ there exists an $\EVI_0$-gradient flow of $\V$ starting from it. 
\item Assumption \ref{hyp2} holds by what we said in Section \ref{subsec:RCD}. 
\end{itemize}

Hence, whenever the dynamical Schr\"odinger problem \eqref{eq:schrodinger_pb} is solvable, the $\Gamma$-convergence results of Section \ref{sec:small_noise_convexity} can be applied. 

As for Assumption \ref{hyp1}, there are at least two cases of interest when its full validity can be verified:
\begin{itemize}
\item if $M=\R^d$ and $U$ is superlinear at $\infty$ (which is the case for $U_m$ defined just below with $m>1$), then by \cite[Theorem 10.4.6]{AmbrosioGigliSavare08} the slope of $\V$ can be represented as
\[
|\partial\V|^2(\mu) = \int_{\R^d} |\nabla U'(\rho)|^2\,\d\mu, \qquad \textrm{if } \mu = \rho\mathcal{L}^d,
\]
and by \cite[Proposition 2.2]{GST09} it is sequentially lower semicontnuous w.r.t.\ narrow and thus $W_p$-topology (so the latter can be used as $\sigma$).
\item if $M$ is compact, by Remark \ref{rmk18} we see that the $W_2$-topology is an admissible candidate for $\sigma$. 
\end{itemize}

We conclude that all our results are applicable in the two situations that we have just described.

\subsubsection*{Novelty and related literature}  

To the best of our knowledge, up to now the dynamical Schr\"odinger problem \eqref{eq:schrodinger_pb} with the slope of a general internal energy in place of the slope of the Boltzmann entropy has been considered only in \cite{GLR18} from a purely formal point of view. Static Monge-Kantorovich problems regularized by means of the R\'enyi entropy or more general internal energies have recently been introduced in \cite{ES18,LMM19,LoM20,MG20} (see also the references therein). Remarkably, \cite{LoM20} establishes the $\Gamma$-convergence of the regularized problems towards the optimal transport one (cf.\ \cite{MG20} where the convergence of the optimal values and minimizers is discussed). However, in \cite{MG20} only bounded costs are considered (the quadratic cost function associated to \eqref{BB} is thus ruled out for non-compact sample spaces), while in \cite{LoM20} the discussion is restricted to sample spaces which are compact subset of $\R^d$. Other questions our paper is concerned with have not been examined in these references. Note also that the issue of the equivalence between static and dynamical formulations is far from being clear at this level of generality. In view of this discussion, in all the applicability situations presented in this section our results are new. 

The case of a (possibly) negatively curved base space $M$ is not discussed since, as already argued above, \cite[Theorem 9.21]{AmbrosioMondinoSavare13} allows to deduce \eqref{eq:EVIlambda} with $\lambda = 0$ only for $K \geq 0$. Moreover, it has recently been proved \cite[Theorem 2.5 and Remark 2.6]{DPMO19} that in the hyperbolic space the porous medium equation cannot be seen as the Wasserstein gradient flow of some $\lambda$-convex functional in the $\EVI$-sense, hence the R\'enyi entropy cannot generate an $\EVI_\lambda$-gradient flow there. 

\subsection{Mean-field Schr\"odinger problem}

In the seminal thought experiment proposed by Schr\"odinger \cite{Schrodinger31, Schrodinger32} the physical system, whose evolution between two subsequent observations has to be determined, consists of independent Brownian particles.
An important generalization has been recently proposed in \cite{BCGL20}, where particles are allowed to interact through a pair potential $W$.
This leads to the so-called Mean Field Schr\"odinger Problem (MFSP henceforth), which can be cast as a metric Schr\"odinger problem by choosing the following setting.

\subsubsection*{Setting}
\begin{itemize} 
\item Let $\X := \mathcal P_2(\R^d)$ be the 2-Wasserstein space over $\R^d$, equipped with the $2$-Wasserstein distance $\sfd:=W_2$. 
\item The role played by the Boltzmann-Shannon relative entropy in the ``classical'' Schr\"odinger problem is here taken by the functional $\V : \X \to \R$ defined (up to a shift by a constant) by
\[
\V(\mu) := \left\{\begin{array}{ll}
\displaystyle{H(\mu\,|\,\mathcal{L}^d) + \int_{\R^d}W * \rho\,\d\mu} & \qquad \text{ if } \mu = \rho\mathcal{L}^d\\
+\infty & \qquad \text{ if } \mu \not\ll \mathcal{L}^d
\end{array}\right.
\]
where $H(\mu\,|\,\mathcal{L}^d)$ is the Boltzmann-Shannon relative entropy of $\mu$ w.r.t.\ the Lebesgue measure $\mathcal{L}^d$, already introduced in Section \ref{subsec:RCD}, and $W$ is the pair potential, describing via convolution the interaction between the particles of the system. On such a potential the following assumptions are made: it is of class $C^2(\R^d,\R)$, is symmetric, i.e.\ $W(x) = W(-x)$ for all $x \in \R^d$, and satisfies the two-sided bound
\[
\Lambda\mathrm{Id}~\geq  \nabla^2 W \geq \lambda\mathrm{Id}
\]
for some $\Lambda, \lambda > 0$ (actually $\lambda\in\R$ is enough, but in \cite{BCGL20} the authors are interested in the ergodic behaviour of MFSP). While the upper bound is technical, the lower one is geometric and crucial.
\item As topology $\sigma$ we shall use the metric topology of $(\X, W_p)$, $1\leq p<2$.
\end{itemize}

Indeed, the fact that the MFSP coincides with \eqref{eq:schrodinger_pb} with the above choice of $(\X,\sfd)$ and $\V$ follows from \cite[Theorem 1.2]{BCGL20} and the fact that the slope of $\V$ is explicitly given by
\[
|\partial\V|^2(\mu) = \left\{\begin{array}{ll}
\displaystyle{\int_{\R^d}|\nabla\log\rho + 2\nabla W * \rho|^2\,\d\mu} & \qquad \text{ if } \mu = \rho\mathcal{L}^d,\,\nabla\log\rho \in L^2_\mu, \\
+\infty & \qquad \text{ otherwise},
\end{array}\right.
\]
cf.\ \cite[Section 1.4.2]{BCGL20} and \cite[Theorem 10.4.13]{AmbrosioGigliSavare08}.

\subsubsection*{Applicability of our results}

In order to see that our abstract metric results hold for this specific example, let us check separately the validity of Setting \ref{setting} and Assumption \ref{hyp1}.
\begin{itemize}
\item Setting \ref{setting} fully holds. As already said in Section \ref{subsec:RCD}, $\X := \mathcal P_2(\R^d)$ is a complete and separable metric space. The lower semicontinuity of $\V$ is easily seen to hold: the relative entropy has already been discussed, whereas the continuity of the convolution term follows from the fact that if $\mu_n \to \mu$ in $\mathcal P_2(\R^d)$, then $\mu_n \otimes \mu_n \to \mu \otimes \mu$ in $\mathcal P_2(\R^{2d})$, cf.\ \cite[Example 9.3.4]{AmbrosioGigliSavare08}.
The fact that $\V$ is proper and the density of its domain are also clear.
Moreover, the assumptions on $W$ guarantee that $\V$ is bounded from below on $W_2$-bounded sets.
As concerns the existence of $\EVI_\lambda$-gradient flows starting from any $\mu \in \X$, this is ensured by \cite[Theorem 11.2.1]{AmbrosioGigliSavare08} in conjunction with \cite[Remark 9.2.5 and Proposition 9.3.5]{AmbrosioGigliSavare08}, granting the $\lambda$-convexity of $\V$ along \emph{generalized} geodesics (see \cite[Definitions 9.2.2 and 9.2.4]{AmbrosioGigliSavare08}).

\item Assumption \ref{hyp1} is fully satisfied as well. Indeed, by the discussion in Section \ref{subsec:RCD} we have the validity of Assumption \ref{hyp2}, so that it only remains to discuss the sequential lower semicontinuity of $|\partial\V|$ w.r.t.\ $W_p$, $1 \leq p < 2$. To this end, one can rely on the previous explicit expression for $|\partial\V|$, on \cite[Proposition 2.2]{GST09}, the fact that $\Delta W$ is continuous and bounded (as a consequence of the boundedness of $\nabla^2 W$) and the regularization properties of the convolution to show that $|\partial\V|$ is sequentially narrowly lower semicontinuous, so that a fortiori it is also sequentially $W_p$-lower semicontinuous. 
\end{itemize}

Hence all the results of Sections \ref{sec:small_noise_convexity} and \ref{sec:derivative_cost} are applicable.

\subsubsection*{Novelty and related literature}

From the novelty standpoint, a first interesting remark is the fact that in \cite{BCGL20} the approach is purely stochastic, while our point of view is completely analytic. For instance, in \cite[Proposition 1.1]{BCGL20} the existence of solutions to MFSP is proved under the same assumptions we have in Proposition \ref{prop:schro_solvable_iff_finite_entropy}, namely $\mu,\nu \in \X$ with $\V(\mu),\V(\nu) < \infty$.
However, already at this basic level the reader may appreciate the difference between the two approaches.

But more than anything else, our abstract results are completely new when specialized to MFSP: indeed, only the ergodic behaviour in the long time regime $\eps \to \infty$ is studied in \cite{BCGL20}, so that the $\Gamma$-convergence results of Section \ref{sec:small_noise_convexity} are entirely novel. The same is true for Section \ref{sec:derivative_cost}, since in \cite{ConTam19} the derivative of the cost associated to MFSP is not investigated nor is the Taylor expansion \eqref{eq:taylor}.

\subsection{Non-linear mobilities}

While in the previous examples the distance $\sfd$ was always the 2-Wasserstein distance, we now turn our attention to the so-called \emph{non-linear mobility} Wasserstein distance, first introduced in \cite{DNS09} as Benamou-Brenier-like generalization of the quadratic Wasserstein distance, and further studied in \cite{CLSS10}.
Making the discussion below completely rigorous would require a tedious and lenghty distinction between various possible structural assumptions on the mobility function $\mathsf m(\rho)$ (cases A and B in \cite{CLSS10}), hence for the sake of presentation we deliberately remain partially informal in this last example.

\subsubsection*{Setting}

\begin{itemize}
\item 
Given a \emph{non-linear mobility} function $\mathsf m:\R^+\to\R^+$ satisfying some structural conditions from \cite{DNS09,CLSS10}, and a convex, smooth, bounded Euclidean domain $\Omega$, the nonlinear Wasserstein distance $W_{\mathsf m}$ is defined on the space of probability measures $\P(\overline\Omega)$ as
\[
W_{\mathsf m}(\mu,\nu) := \inf\left\{\int_0^1\int_{\R^d}|v_t(x)|^2\mathsf m(\rho_t(x))\d x \d t\right\},
\]
where the infimum runs over all distributional solutions of the \emph{non-linear} continuity equation $$\partial_t\rho_t + {\rm div}(v_t\mathsf m(\rho_t))=0$$ with the constraints $\rho_0\mathcal{L}^d = \mu$, $\rho_1\mathcal{L}^d = \nu$ and $\supp(\rho_t) \subset \overline{\Omega}$. 
We consider the metric\footnote{
	Note that, depending on the structural assumptions on $\mathsf m$ and the particular measures $\mu,\nu$, it can happen that $W_{\mathsf m}(\mu,\nu)=+\infty$.
	In this case, instead of taking the whole $X=\P(\overline\Omega)$ as a \emph{pseudo}-metric space, one should rather work on the finite components $X:=\P[\nu]=\{\mu\in\P(\overline\Omega)\mbox{ s.t. }W_{\mathsf m}(\mu,\nu)<\infty$ for fixed $\nu$.
	This becomes indeed a complete and separable metric space \cite[Prop. 3.2]{CLSS10} but for simplicity we shall ignore this subtle issue.
} space $(\X,\sfd) := (\P(\overline{\Omega}),W_{\mathsf m})$.

\item
For the entropy functional $\V(\mu)$ we consider the internal energy $\int_\Omega U(\rho)\d\mathcal{L}^d$  defined in Section \ref{subsec:renyi} with $M=\Omega$ and $\m = \mathcal{L}^d|_{\Omega}$.
We assume that the function $U$ defining $\V$ is non-negative and satisfies the \emph{generalized McCann condition} $GMC(\mathsf m,d)$ from \cite[Definition 4.5]{CLSS10}.

\item 
The classical narrow convergence of measures plays the role of $\sigma$.
\end{itemize}

\subsubsection*{Applicability of our results}

Let us first discuss the validity of our main assumptions.
\begin{itemize}
\item
First of all, completeness and separability for our assumption \ref{item:complete_sep} are known from \cite{DNS09,CLSS10}.
Secondly, by \cite[Section 4.1]{CLSS10} $\V$ is narrowly lower semicontinuous, and since by \cite[Theorem 5.5]{DNS09} the $W_{\mathsf m}$-topology is stronger than the narrow one, the lower semicontinuity of $\V$ w.r.t.\ $W_{\mathsf m}$ follows. 
The non-negativity of $U$ implies that $\V$ is (globally) bounded from below.
The density of the domain $D(\V)$ is ensured by \cite[Corollary 4.11]{CLSS10}, hence
our condition \ref{item:V_locally_bounded_below} fully holds.
As regards our more fundamental assumption \ref{item:generate_lambda_flow}, the generation of an $\EVI_\lambda$-flow is exactly the purpose of \cite{CLSS10} for $\lambda=0$ under the generalized McCann condition.
As a consequence our Setting~\ref{setting} is fully applicable here.

%
%
\item
We come now to the more delicate Assumption \ref{hyp1}, whose validity would readily make all our results of Sections \ref{sec:small_noise_convexity} and \ref{sec:derivative_cost} rigorously applicable. 
First of all, the $\sigma$-compactness of bounded sets holds, simply because $\P(\overline{\Omega})$ is always narrowly compact.
The narrow lower semicontinuity of the distance $W_{\mathsf m}$ is known from \cite[Theorem 5.6]{DNS09}.
As for the lower semicontinuity of the slope  $|\partial\V|$, this is where our discussion becomes informal:
contrarily to the previous examples, and despite the Euclidean structure and the additional compactness of the domain, we are not aware of a rigorous explicit expression for the slope.
Yet from a purely formal point of view, and given the pseudo-Riemannian structure induced by \cite[Eq. {3.2}]{CLSS10}, this metric slope is clearly expected to be
\begin{equation}
\label{eq:slope-nonlinear}
|\partial\V|^2(\rho)
=\int_\Omega |\nabla U'(\rho)|^2 \mathsf  m(\rho)\,\d x
=\int_\Omega\frac{|\nabla P(\rho)|^2}{\mathsf m(\rho)}\,\d x,
\end{equation}
a generalized Fisher information.
Here the \emph{pressure} $P$ is defined as $P(r)=\int_0^r U''(s)\mathsf m(s)\,\d s$.
Under additional conditions on $U,\mathsf m$, one may be able to check by hand the narrow lower semicontinuity of \eqref{eq:slope-nonlinear}, thus validating our Assumption \ref{hyp1}.
The stringent computation of the metric slope $|\partial\V|(\rho)$ is actually listed as an open problem in \cite[Section 7]{CLSS10}, and this issue is the main obstacle to making the particular application of our abstract results to nonlinear mobilities completely rigorous.
\end{itemize} 

Regardless of the issue that Assumption \ref{hyp1} cannot be completely validated here, let us point out that the above discussion does fully entail Assumption \ref{hyp2}, thus as already discussed our main $\Gamma$-convergence results (Theorem \ref{t:gconv1} and Corollary \ref{c:gc}) would hold, provided one could establish the existence of minimizers for the $\eps$-problems.

Nonetheless, in this setting and from a PDE perspective, \eqref{eq:schrodinger_pb} should look like
\begin{equation}
\label{eq:Scheps_NLmobility}
\inf\limits_{\rho,v} \bigg\{ \frac{1}{2}\int_0^1\int_\Omega |v_t|^2\mathsf m(\rho_t)\,\d x\d t + \frac{\eps^2}{2}\int_0^1 \int_\Omega\frac{|\nabla P(\rho)|^2}{\mathsf m(\rho)}\,\d x\d t \bigg\},
\end{equation}
with the infimum running over all distributional solutions of the non-linear continuity equation.

\subsubsection*{Novelty and related literature}

To the best of our knowledge, the approximation of the $W_{\mathsf m}$-geodesic problem by the dynamical Schr\"odinger-like problem \eqref{eq:Scheps_NLmobility} has never been considered before in the literature.
Therefore, the $\Gamma$-convergence results of Section \ref{sec:small_noise_convexity} recast in the current setting are entirely new, and a fortiori so are the derivative of the entropic cost and its Taylor expansion at $\eps=0$ discussed in Section \ref{sec:derivative_cost}.
Instead, we would like to stress once more that identity \eqref{eq:slope-nonlinear} is purely formal, and turning it into a rigorous statement falls out of scope of the paper.
Establishing a rigorous connection between our abstract metric framework and the hands-on PDE formulation \eqref{eq:Scheps_NLmobility} is thus an interesting question, which will require additional technical work and is left for future developments.


\subsection*{Acknowledgments}

LM wishes to thank Jean-Claude Zambrini for numerous and fruitful discussions on the Schr\"odinger problem, and acknowledges support from the Portuguese Science Foundation through FCT project PTDC/MAT-STA/22812/2017 \emph{Schr\"oMoka}.
LT is grateful to Nicola Gigli for useful comments, and acknowledges financial support from FSMP Fondation Sciences Math\'ematiques de Paris.
DV was partially supported by the FCT projects UID/MAT/00324/2020 and\\ PTDC/MAT-PUR/28686/2017.



\end{document}